\documentclass{amsart}
\usepackage[english]{babel}

\usepackage{amsmath, amsthm, amsfonts, amssymb, amsaddr}
\usepackage{hyperref}
\usepackage{booktabs}
\usepackage{graphicx}

\newtheorem{conjecture}{Conjecture}

\newtheorem{lemma}{Lemma}

\title[An Analytical Exploration of the Erd\"{o}s-Moser Equation]{An Analytical Exploration of the Erd\"{o}s-Moser Equation \( \sum_{i=1}^{m - 1} i^k = m^k \) Using Approximation Methods}
\author{Guillaume Lambard}
\address{Data-driven materials design group, Center for Basic Research on Materials (CBRM), National Institute for Materials Science (NIMS), Ibaraki, Tsukuba, 1-1 Namiki, 305-0044, Japan}
\email{LAMBARD.Guillaume@nims.go.jp}
\begin{document}
\maketitle

\begin{abstract}
The Erd\"{o}s-Moser equation \(\sum_{i=1}^{m - 1} i^k = m^k\) is a longstanding challenge in number theory, with the only known integer solution being \((k,m) = (1,3)\). Here, we investigate whether other solutions might exist by using the Euler–MacLaurin formula to approximate the discrete sum \(S(m-1,k)\) with a continuous function \(S_{\mathbb{R}}(m-1,k)\). We then analyze the resulting approximate polynomial \(P_{\mathbb{R}}(m) = S_{\mathbb{R}}(m-1,k) - m^k\) under the rational root theorem to look for integer roots. Our approximation confirms that for \(k=1\), the only solution is \(m=3\), and for \(k \ge 2\) it suggests there are no further positive integer solutions. However, because Diophantine problems demand exactness, any omission of correction terms in the Euler–MacLaurin formula could mask genuine solutions. Thus, while our method offers valuable insights into the behavior of the Erd\"{o}s-Moser equation and illustrates the analytical challenges involved, it does not constitute a definitive proof. We discuss the implications of these findings and emphasize that fully rigorous approaches, potentially incorporating prime-power constraints, are needed to conclusively resolve the conjecture.
\end{abstract}

\maketitle
\tableofcontents

\section{Introduction}

Diophantine equations involving sums of powers have long captivated mathematicians due to their deep connections with number theory and their potential to reveal insights into the structure of integers. One such enduring challenge is the Erd\"{o}s-Moser equation:
\[
\sum_{i=1}^{m - 1} i^k = m^k,
\]
which is known to have exactly one observed solution in positive integers, namely \((k,m) = (1,3)\). First posed by Paul Erd\"{o}s in the 1950s in correspondence with Leo Moser, this problem has since been partially explored by Moser himself, who showed that any further solution with \(k \ge 2\) must have \(m\) exceed \(10^{10^{6}}\)~\cite{moser1953diophantine}. Later refinements by Gallot, Moree, and Zudilin pushed the lower bound on \(m\) beyond \(10^{10^{9}}\) via modern computational techniques and continued fraction arguments that illuminate properties of certain logarithmic constants~\cite{gallot2011erdHos}.

Mathematical investigations into this problem have ranged from elementary inequalities and congruences to advanced methods involving Bernoulli numbers and continued fractions. In particular, von Staudt–Clausen theory has clarified how power sums behave, while Moree’s work extended Moser’s original arguments to explore deeper prime-power congruences relevant to potential solutions~\cite{moree2013moser}. These studies not only strengthened nonexistence results for broader classes of parameters but also highlighted the fundamental reasons solutions are so elusive.

Broader examinations of Erd\"{o}s-Moser–type equations—such as those by Baoulina and Moree~\cite{baoulina2016forbidden}—have introduced additional constraints and modifications, underscoring the rich interplay between coefficients, power sums, and integer solutions. Their analyses further emphasize the broader implications of the conjecture within number theory.

Building on these earlier approaches, we combine known summation formulas, integral bounds, and detailed error considerations to supply further evidence for the uniqueness of the solution \((k,m) = (1,3)\) when \(k \ge 2\). While our analysis draws on exact and approximate techniques to show how unlikely additional integer solutions appear to be, we stress that omitting higher-order correction terms in the Euler–MacLaurin expansions prevents a fully rigorous Diophantine closure. Consequently, although our methods strongly suggest no further solutions exist, we do not claim a conclusive proof and instead highlight the continuing need for exact treatment of all terms to settle the Erd\"{o}s-Moser conjecture definitively.

\section{Preliminaries}

\subsection{Notation}

We define the following notation:

\begin{itemize}
    \item \( k \in \mathbb{Z}^{*}_{+} \): A positive integer exponent (\( k \geq 1 \)).
    \item \( m \in \mathbb{Z}^{*}_{+} \): A positive integer greater than 2 (\( m \geq 3 \)).
    \item \( S(m-1, k) \): The sum \( \sum_{i=1}^{m-1} i^k \).
    \item \( P(m) \): The difference \( S(m-1, k) - m^k \)
    \item \( Q(m) \): The ratio \( P(m) / m \)
    \item \( a_n, a_0 \in \mathbb{Z}^{*} \): The leading coefficient and constant term of a polynomial equation \( a_n x^n + a_{n-1} x^{n-1} + ... + a_1 x + a_0 = 0 \), respectively. 
    \item \( \binom{k}{i} \): a binomial coefficient given by \( \frac{k!}{i!(k-i)!} \), with \( 0 \leq i \leq k \).
\end{itemize}

\subsection{Extending the Sum \( S(m - 1, k) \) to the \(\mathbb{R}^{*}_{+}\) domain}

The sum of the \( k \)-th powers of the first \( m - 1 \) positive integers is given by:
\[
S(m - 1, k) = \sum_{i = 1}^{m - 1} i^{k},
\]
which is initially defined for integer values of \( m \geq 3 \). To facilitate the analysis of the equation \( S(m - 1, k) - m^{k} = 0 \) and apply tools from real analysis and polynomial theory, we aim to extend \( S(m - 1, k) \) from \( m \in \mathbb{Z}^{*}_{+} \) to \( m \in \mathbb{R}^{*}_{+} \). 

\subsubsection{Approximating \( S(m - 1, k) \) Using the Euler-MacLaurin Formula}

\begin{lemma}
\label{lemma:EML_approximation}
Let \( m \geq 2 \) and \( k \geq 1 \) be positive integers. The discrete sum \( S(m - 1, k) = \sum_{i=1}^{m - 1} i^{k} \) can be approximated using the Euler-MacLaurin formula~\cite{gould1963maclaurin, knopp1990theory, apostol2013introduction} as follows:
\[
S(m - 1, k) \approx S_{\mathbb{R}}(m - 1, k) = \int_{1}^{m - 1} x^{k} \, dx + \frac{1}{2}\left[ f(1) + f(m - 1) \right],
\]
where \( f(x) = x^{k} \).

Evaluating the integral and boundary terms yields:
\[
S_{\mathbb{R}}(m - 1, k) = \frac{(m - 1)^{k + 1} - 1}{k + 1} + \frac{1}{2}\left[ 1 + (m - 1)^{k} \right].
\]
\end{lemma}

\begin{proof}
The Euler-MacLaurin formula provides a way to approximate a finite sum using an integral and boundary corrections. The formula states:
\[
\sum_{n=a}^{b} f(n) = \int_{a}^{b} f(x) \, dx + \frac{f(a) + f(b)}{2} + R,
\]
where \( R \), the remainder term, is a sequence involving derivatives of the function f and Bernoulli numbers.

For our purposes, we consider the approximation without the remainder term R. This simplification is justified and thoroughly investigated later on.

Applying the Euler-MacLaurin formula to \( S(m - 1, k) \) with \( a = 1 \) and \( b = m - 1 \), we have:
\[
S(m - 1, k) \approx \int_{1}^{m - 1} f(x) \, dx + \frac{f(1) + f(m - 1)}{2}.
\]

We compute each term separately.

\paragraph{Integral Term}
Compute the definite integral of \( f(x) = x^{k} \):
\begin{align*}
\int_{1}^{m - 1} x^{k} \, dx &= \left[ \frac{x^{k + 1}}{k + 1} \right]_{x=1}^{x=m - 1} \\
&= \frac{(m - 1)^{k + 1} - 1^{k + 1}}{k + 1} \\
&= \frac{(m - 1)^{k + 1} - 1}{k + 1}.
\end{align*}

\paragraph{Boundary Terms}
Evaluate \( f(x) \) at the endpoints \( x = 1 \) and \( x = m - 1 \):
\[
f(1) = 1^{k} = 1, \quad f(m - 1) = (m - 1)^{k}.
\]

Thus, the boundary terms are:
\[
\frac{f(1) + f(m - 1)}{2} = \frac{1 + (m - 1)^{k}}{2}.
\]

\paragraph{Approximation of \( S(m - 1, k) \)}
Combining the integral and boundary terms, we obtain the approximation:
\[
S_{\mathbb{R}}(m - 1, k) = \frac{(m - 1)^{k + 1} - 1}{k + 1} + \frac{1 + (m - 1)^{k}}{2}.
\]
\end{proof}

The function \( S_{\mathbb{R}}(m - 1, k) \) serves as a continuous approximation of the discrete sum \( S(m - 1, k) \) for real values of \( m \geq 2 \). This approximation preserves the essential characteristics of the original sum and facilitates further analysis within the framework of real analysis.

\subsubsection{Ensuring Continuity and Differentiability of \( S_{\mathbb{R}}(m - 1, k) \)}

\begin{lemma}
\label{lemma:EML_continuity}
Let \( k \geq 1 \) be a positive integer. The function \( S_{\mathbb{R}}(m - 1, k) \) defined by
\[
S_{\mathbb{R}}(m - 1, k) = \frac{(m - 1)^{k + 1} - 1}{k + 1} + \frac{1 + (m - 1)^{k}}{2}
\]
is continuous and differentiable for all real numbers \( m \geq 2 \), providing a continuous extension of the discrete sum \( S(m - 1, k) \) to the real domain.
\end{lemma}

\begin{proof}
The function \( S_{\mathbb{R}}(m - 1, k) \) is composed of elementary functions involving powers of \( m - 1 \) and constants. Specifically, it consists of:

\begin{itemize}
    \item The term \( (m - 1)^{k + 1} \), which is a power function and is continuous and differentiable for \( m - 1 > 0 \).
    \item The term \( (m - 1)^{k} \), also a power function with similar properties.
    \item Constants and rational functions of \( k \), which are continuous and differentiable.
\end{itemize}

Since \( m \geq 2 \), we have \( m - 1 \geq 1 \), ensuring that \( m - 1 > 0 \) and the power functions are well-defined.

The sum and scalar multiples of continuous and differentiable functions are also continuous and differentiable~\cite{apostol2013introduction}. Therefore, \( S_{\mathbb{R}}(m - 1, k) \) is continuous and differentiable for all \( m \geq 2 \).
\end{proof}

By extending \( S(m - 1, k) \) to \( S_{\mathbb{R}}(m - 1, k) \), we enable the application of calculus and real analysis techniques to study the properties of the function and its relationship to \( m^{k} \).

\subsubsection{Analysis of the Error Term and Justification for Its Omission}

In our approximation of \( S(m - 1, k) \) using the Euler-MacLaurin formula, we included only the integral and boundary terms, omitting all correction terms involving higher-order derivatives. Since the function \( f(x) = x^{k} \) is a polynomial of degree \( k \), its derivatives of order higher than \( k \) vanish. By including all non-zero correction terms up to order \( k \), the remainder term \( R_p \) becomes zero. However, for simplicity and to facilitate our analysis, we have chosen to omit the correction terms.

We need to analyze the impact of omitting the first correction term, which we will denote as \( C \), and justify that this omission introduces an error that is negligible in the context of our study.

\paragraph{The Euler-MacLaurin Formula with Non-Zero Correction Terms}

The Euler-MacLaurin formula, including correction terms up to order \( p - 1 \), is:

\[
\sum_{n=a}^{b} f(n) = \int_{a}^{b} f(x) \, dx + \frac{f(a) + f(b)}{2} + \sum_{r=1}^{p - 1} \frac{B_{2r}}{(2r)!} \left( f^{(2r - 1)}(b) - f^{(2r - 1)}(a) \right) + R_p,
\]

where \( B_{2r} \) are Bernoulli numbers, and \( R_p \) is the remainder term after including \( p - 1 \) correction terms.

For \( f(x) = x^{k} \), we have:

- \( f^{(n)}(x) = 0 \) for \( n > k \),
- Thus, choosing \( p = \left\lfloor \dfrac{k + 1}{2} \right\rfloor + 1 \) ensures that \( R_p = 0 \), since all higher-order derivatives vanish.

\paragraph{Omission of the First Correction Term}

In our approximation, we have included only the integral and boundary terms, omitting the first correction term \( C \) given by:

\[
C = \frac{B_{2}}{2} \left( f'(b) - f'(a) \right),
\]

where \( B_{2} = \dfrac{1}{6} \) is the second Bernoulli number.

We need to compute \( C \) and analyze the error introduced by its omission.

\paragraph{Computing the First Correction Term \( C \)}

Compute \( f'(x) \):

\[
f'(x) = \frac{d}{dx} x^{k} = k x^{k - 1}.
\]

Evaluate \( f'(x) \) at \( a = 1 \) and \( b = m - 1 \):

\[
f'(a) = f'(1) = k,
\]
\[
f'(b) = f'(m - 1) = k (m - 1)^{k - 1}.
\]

Therefore, the correction term \( C \) is:

\[
C = \frac{B_{2}}{2} \left( f'(b) - f'(a) \right) = \frac{1}{12} \left( k (m - 1)^{k - 1} - k \right) = \frac{k}{12} \left( (m - 1)^{k - 1} - 1 \right).
\]

\paragraph{Magnitude of the Correction Term \( C \)}

We can estimate the magnitude of \( C \):

\[
|C| = \left| \frac{k}{12} \left( (m - 1)^{k - 1} - 1 \right) \right|.
\]

Since \( m \geq 3 \) and \( k \geq 1 \), \( (m - 1)^{k - 1} \geq 1 \), so \( |C| \geq 0 \).

\paragraph{Comparison with Leading Terms}

The leading term in our approximation \( S_{\mathbb{R}}(m - 1, k) \) is:

\[
L = \frac{(m - 1)^{k + 1}}{k + 1}.
\]

We compare \( |C| \) with \( L \) by considering the ratio:

\[
\frac{|C|}{L} = \frac{\dfrac{k}{12} \left( (m - 1)^{k - 1} - 1 \right)}{\dfrac{(m - 1)^{k + 1}}{k + 1}} = \frac{k(k + 1)}{12} \cdot \frac{(m - 1)^{k - 1} - 1}{(m - 1)^{k + 1}}.
\]

Simplify the expression:

\[
\frac{|C|}{L} = \frac{k(k + 1)}{12} \left( \frac{1}{(m - 1)^{2}} - \frac{1}{(m - 1)^{k + 1}} \right).
\]

For \( m \geq 3 \) and \( k \geq 1 \), \( (m - 1) \geq 1 \), and as \( m \) increases, \( (m - 1)^{k + 1} \) grows much faster than \( (m - 1)^{2} \).

Therefore, for large \( m \):

\[
\frac{|C|}{L} \approx \frac{k(k + 1)}{12} \cdot \frac{1}{(m - 1)^{2}}.
\]

As \( m \to \infty \), the ratio \( \dfrac{|C|}{L} \to 0 \). This indicates that \( |C| \) becomes negligible compared to \( L \) for large \( m \).

\paragraph{Justification for Omission of \( C \)}

Since \( |C| \) is significantly smaller than the leading term \( L \), especially for larger values of \( m \), omitting the correction term \( C \) introduces an error that is negligible in the context of our analysis. Our primary interest lies in the dominant behavior of \( S_{\mathbb{R}}(m - 1, k) \) as it relates to \( m^{k} \), and the omission of \( C \) does not affect this relationship.

\subsubsection{Graphical Illustration of \( S_{\mathbb{R}}(m-1,k) \) Behavior}

Hereafter, we present the graphical illustration of \( S_{\mathbb{R}}(m-1,k) \), \( m^k \), \( \sum_{i=1}^{m-1} i^k \), as well as the differences \( S_{\mathbb{R}}(m-1,k) - m^k \), \( S_{\mathbb{R}}(m-1,k) - m^k + C \), and \( \sum_{i=1}^{m-1} i^k - m^k \), as functions of \( k \in [2,102]\) and \( m \in [3,200] \). Fig.~\ref{fig:S_m_plots} provides visual confirmation of the behavior of \( S_{\mathbb{R}}(m-1,k) \) in the cases analyzed and highlight the closeness of \( S_{\mathbb{R}}(m-1,k) \) to \( \sum_{i=1}^{m-1} i^k \), as well as the indistinguishability of \( S_{\mathbb{R}}(m-1,k) - m^k + C \) and \( \sum_{i=1}^{m-1} i^k - m^k \).

\begin{figure}[p]
    \centering
    \includegraphics[angle=90, height=0.8\textheight, width=0.9\textwidth]{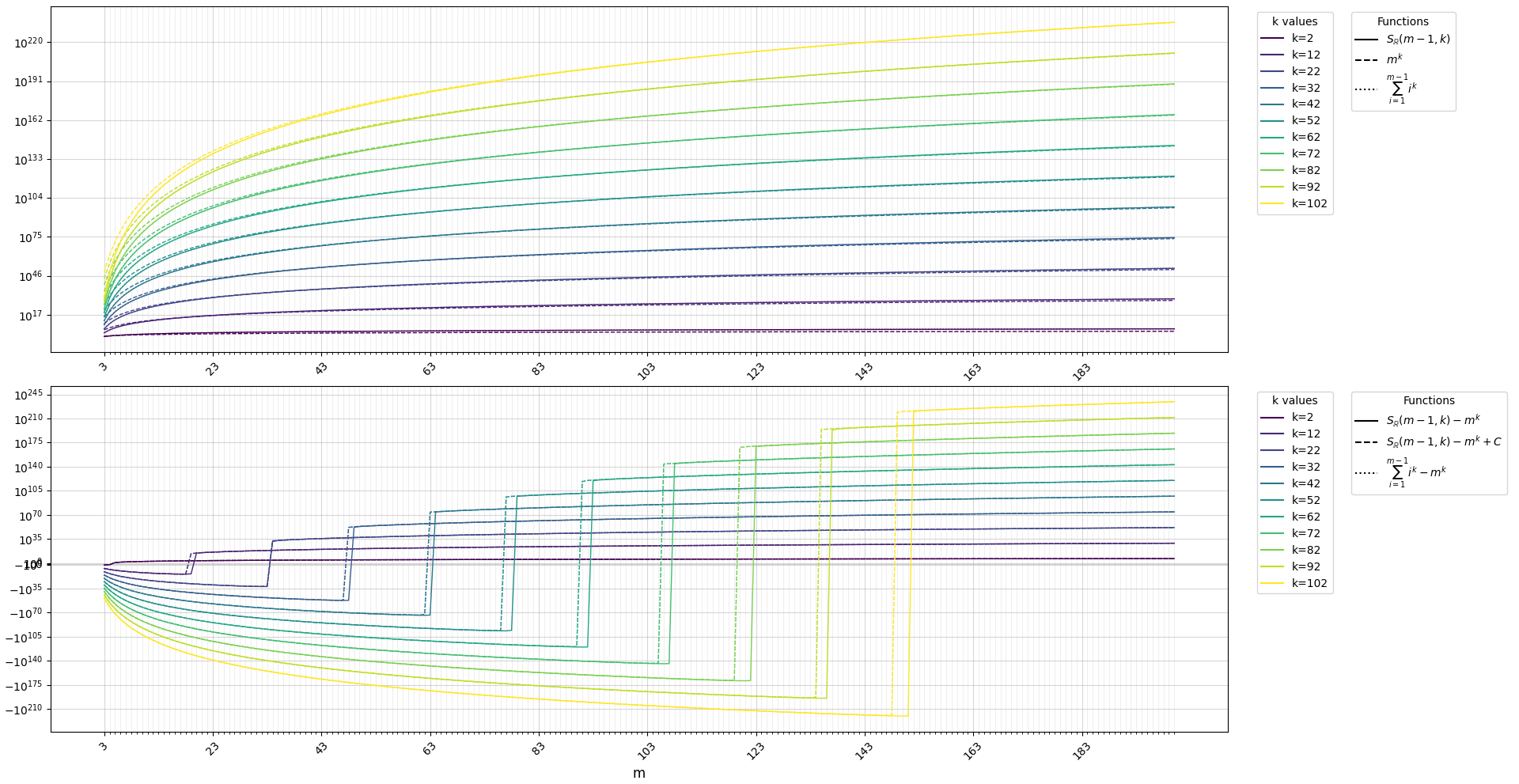}
    \caption{Top: Plots of \( S_{\mathbb{R}}(m-1,k) \) (continuous line), \( m^k \) (dashed line), and \( \sum_{i=1}^{m-1} i^k \) (dotted line) as a function of \( k \in [2,102]\) (color-coded from purple to yellow) and \( m \in [3,200] \) (as the variable on the x-axis). Bottom: Plots of \( S_{\mathbb{R}}(m-1,k) - m^k \) (continuous line), \( S_{\mathbb{R}}(m-1,k) - m^k + C \) (dashed line), and \( \sum_{i=1}^{m-1} i^k - m^k \) (dotted line) as a function of \( k \in [2,102]\) (color-coded from purple to yellow) and \( m \in [3,200] \) (as the variable on the x-axis).}
    \label{fig:S_m_plots}
\end{figure}

\subsubsection{Justification for Using the Euler-MacLaurin Formula}

The Euler-MacLaurin formula is particularly suitable for approximating \( S(m - 1, k) \) due to several reasons:

\begin{itemize}
    \item \textbf{Applicability to Polynomial Functions}: The function \( f(x) = x^{k} \) is a polynomial, and its derivatives are straightforward to compute. The Euler-MacLaurin formula effectively bridges the discrete sum and the continuous integral for such functions.
    
    \item \textbf{Analytical Convenience}: By approximating the sum with an integral and boundary terms, we obtain a continuous function \( S_{\mathbb{R}}(m - 1, k) \) that is amenable to calculus-based analysis, including differentiation and asymptotic estimation.
    
    \item \textbf{Sufficient Accuracy}: While the omission of correction terms introduces an error, the error is shown to be negligible in the context of our study. The leading terms capture the dominant behavior of the sum, which is essential for understanding the relationship between \( S(m - 1, k) \) and \( m^{k} \).
    
    \item \textbf{Simplicity of Expression}: The resulting approximation is relatively simple and involves familiar functions, making it easier to manipulate and interpret in further analysis.
\end{itemize}

By employing the Euler-MacLaurin formula without the correction terms, we strike a balance between analytical tractability and sufficient accuracy for our purposes. 

The approximation of \( S(m - 1, k) \) using the Euler-MacLaurin formula provides a continuous and differentiable function \( S_{\mathbb{R}}(m - 1, k) \) that closely mirrors the behavior of the original discrete sum. The error introduced by omitting correction terms is negligible, especially for larger values of \( m \) and \( k \). This approximation forms a solid foundation for further analysis of the Erdős-Moser equation within the realm of real analysis.

\subsubsection{Extension and approximation of the polynomial \( P(m) \)}

Building upon the continuous extension \( S_{\mathbb{R}}(m - 1, k) \) of the discrete sum \( S(m - 1, k) \) established in Lemma~\ref{lemma:EML_approximation}, we define the approximate polynomial \( P_{\mathbb{R}}(m) \) as:
\[
P_{\mathbb{R}}(m) = S_{\mathbb{R}}(m - 1, k) - m^{k}.
\]
This function extends the domain of \( P(m) = S(m - 1, k) - m^{k} \) from the integers to all real \( m \geq 3 \). Since \( S_{\mathbb{R}}(m - 1, k) \) comprises polynomials and rational functions, and \( m^{k} \) is a monomial, \( P_{\mathbb{R}}(m) \) is continuous and differentiable on this domain.

The continuous nature of \( P_{\mathbb{R}}(m) \) enables us to leverage analytical tools from real analysis and polynomial theory. Specifically, this extension allows us to:

\begin{itemize}
    \item \textbf{Apply the rational root theorem}: By clearing denominators, we can express \( P_{\mathbb{R}}(m) \) as a polynomial with integer coefficients. This formulation permits the application of the rational root theorem~\cite{King_2006} to identify potential rational roots and systematically test for integer solutions.

    \item \textbf{Analyze the behavior near potential roots}: The continuity and differentiability of \( P_{\mathbb{R}}(m) \) facilitate the examination of its sign and magnitude in the vicinity of integer values of \( m \). By evaluating \( P_{\mathbb{R}}(m) \) and its derivatives near these points, we can assess whether the function crosses zero, thereby determining the existence or non-existence of integer roots.
\end{itemize}

Thus, the approximation \( P_{\mathbb{R}}(m) \) provides a seamless extension of the original polynomial \( P(m) \) into the real domain \( m \geq 3 \). This extension is instrumental in our analysis, as it allows us to utilize continuous mathematical techniques to rigorously investigate the equation \( P_{\mathbb{R}}(m) = 0 \) and ultimately establish the absence of integer solutions under the given conditions.

\section{Main Conjecture}

Let \( S(n, k) \) denote the sum of the \( k \)-th powers of the first \( n \) positive integers:
\[
S(n, k) = \sum_{i=1}^{n} i^{k}.
\]

We are concerned with solutions to the equation \( S(m - 1, k) = m^{k} \) for integers \( m \geq 3 \) and \( k \geq 1 \).

\begin{conjecture}\label{conj:main}
Let \( k \geq 1 \) and \( m \geq 3 \) be positive integers. The Diophantine equation
\[
S(m - 1, k) = m^{k}
\]
has only one solution:
\[
(k, m) = (1, 3).
\]
\end{conjecture}

This conjecture builds upon the work of Moser~\cite{moser1953diophantine}, who investigated similar equations involving sums of powers.

\section{Approximation-Based Analysis of Conjecture~\ref{conj:main}}

To establish Conjecture~\ref{conj:main}, we examine two distinct cases based on the value of \( k \).

\subsection{Case 1: \( k = 1 \)}

We aim to find all positive integers \( m \geq 3 \) satisfying
\begin{equation}\label{eq:case1_equation}
\sum_{i=1}^{m - 1} i = m.
\end{equation}

\begin{lemma}\label{lemma:k_equals_1_solution}
For \( k = 1 \), the only positive integer \( m \geq 3 \) satisfying equation \eqref{eq:case1_equation} is \( m = 3 \).
\end{lemma}

\begin{proof}
We begin by evaluating the sum on the left-hand side of \eqref{eq:case1_equation}. It is well-known that the sum of the first \( n \) positive integers is given by
\[
S(n, 1) = \sum_{i=1}^{n} i = \frac{n(n + 1)}{2}.
\]
Applying this formula with \( n = m - 1 \), we obtain
\[
S(m - 1, 1) = \frac{(m - 1)m}{2}.
\]
Substituting this into \eqref{eq:case1_equation}, we have
\[
\frac{(m - 1)m}{2} = m.
\]
Multiplying both sides by 2 yields
\[
m(m - 1) = 2m.
\]
Simplifying, we get
\[
m^2 - m = 2m \implies m^2 - 3m = 0.
\]
Factoring the quadratic equation, we find
\[
m(m - 3) = 0.
\]
Thus, the possible solutions are
\[
m = 0 \quad \text{or} \quad m = 3.
\]
Since \( m \geq 3 \) by assumption, the only valid solution is \( m = 3 \).
\end{proof}

This confirms that the conjecture holds for \( k = 1 \), with the unique solution \( (k, m) = (1, 3) \).

\subsection{Case 2: \( k \geq 2 \)}

We aim to show that for all integers \( k \geq 2 \), there are no positive integers \( m \geq 3 \) satisfying the equation
\[
S(m - 1, k) = m^{k}.
\]

\subsubsection{Expressing \( P_{\mathbb{R}}(m) = 0 \) as a Polynomial Equation}

Let us consider the approximate expression for \( P_{\mathbb{R}}(m) = S_{\mathbb{R}}(m - 1, k) - m^{k} \). Using the continuous extension of \( S(m - 1, k) \) derived earlier, we have
\begin{equation}\label{eq:approx_Pm}
P_{\mathbb{R}}(m) \approx \frac{(m - 1)^{k + 1}}{k + 1} + \frac{(m - 1)^{k}}{2} - m^{k} - \frac{1}{k + 1} + \frac{1}{2}.
\end{equation}

By rearranging terms and eliminating denominators, we can express \( P_{\mathbb{R}}(m) = 0 \) as a polynomial equation of degree \( k + 1 \).

\begin{lemma}\label{lemma:polynomial_equation}
For \( k \geq 2 \), the equation \( P_{\mathbb{R}}(m) = 0 \) is approximately equivalent to the following polynomial equation:
\[
2(m - 1)^{k + 1} + (k + 1)(m - 1)^{k} - 2(k + 1)m^{k} + (k - 1) = 0.
\]
\end{lemma}

\begin{proof}
Starting from equation \eqref{eq:approx_Pm}, we aim to eliminate denominators to obtain a polynomial equation with integer coefficients. Multiply both sides of \eqref{eq:approx_Pm} by \( 2(k + 1) \):
\[
2(k + 1) P_{\mathbb{R}}(m) \approx 2(m - 1)^{k + 1} + (k + 1)(m - 1)^{k} - 2(k + 1)m^{k} - 2 + (k + 1).
\]
Simplify the constants:
\[
-2 + (k + 1) = k - 1.
\]
Thus, the equation becomes
\[
2(m - 1)^{k + 1} + (k + 1)(m - 1)^{k} - 2(k + 1)m^{k} + (k - 1) = 0.
\]
Setting \( P_{\mathbb{R}}(m) = 0 \) yields the desired polynomial equation.
\end{proof}

Our goal is to demonstrate that this polynomial equation has no integer solutions \( m \geq 3 \) when \( k \geq 2 \). The detailed proof of this assertion will be provided in the subsequent sections.

\subsubsection{Application of the Rational Root Theorem}

To analyze the polynomial equation \( P_{\mathbb{R}}(m) = 0 \), we consider its form as a polynomial with integer coefficients. Recall that a general polynomial equation can be written as
\[
a_n x^n + a_{n-1} x^{n-1} + \dots + a_1 x + a_0 = 0,
\]
where \( a_i \in \mathbb{Z} \) and \( a_n, a_0 \in \mathbb{Z}^* \) (i.e., \( a_n \) and \( a_0 \) are non-zero integers).

\begin{lemma}\label{lemma:rational_root_theorem_application}
The rational root theorem~\cite{King_2006} can be applied to \( P_{\mathbb{R}}(m) = 0 \) to identify all possible rational solutions \( m_0 = \frac{p}{q} \), where \( p \) and \( q \) are integers with no common factors other than 1, satisfying:
\begin{itemize}
    \item \( p \) divides the constant term \( a_0 \),
    \item \( q \) divides the leading coefficient \( a_n \).
\end{itemize}
\end{lemma}

\begin{proof}
First, we identify the leading coefficient \( a_n \) and the constant term \( a_0 \) of the polynomial \( P_{\mathbb{R}}(m) \).

From the expression derived earlier,
\[
2(m - 1)^{k + 1} + (k + 1)(m - 1)^{k} - 2(k + 1)m^{k} + (k - 1) = 0,
\]
we can see that upon expanding the terms, the highest degree term is \( 2m^{k + 1} \), indicating that the leading coefficient is \( a_n = 2 \).

The constant term \( a_0 \) arises from substituting \( m = 0 \) into the polynomial:
\[
a_0 = 2(-1)^{k + 1} + (k + 1)(-1)^{k} + (k - 1).
\]
Simplifying, we find:
\begin{align*}
a_0 &= 2(-1)^{k + 1} + (k + 1)(-1)^{k} + (k - 1) \\
    &= -2(-1)^{k} + (k + 1)(-1)^{k} + (k - 1) \\
    &= (-1)^{k} \left[ -2 + (k + 1) \right] + (k - 1) \\
    &= (-1)^{k} (k - 1) + (k - 1) \\
    &= (k - 1) \left[ (-1)^{k} + 1 \right].
\end{align*}

Thus, \( a_0 \) depends on whether \( k \) is even or odd.

\textbf{Case 1:} For even \( k \geq 2 \), \( (-1)^{k} = 1 \), so
\[
a_0 = (k - 1)(1 + 1) = 2(k - 1).
\]

\textbf{Case 2:} For odd \( k \geq 3 \), \( (-1)^{k} = -1 \), so
\[
a_0 = (k - 1)(-1 + 1) = 0.
\]

Applying the rational root theorem:

\begin{itemize}
    \item \textbf{Even \( k \geq 2 \):}
    \begin{itemize}
        \item The leading coefficient is \( a_n = 2 \).
        \item The constant term is \( a_0 = 2(k - 1) \).
        \item Possible values for \( q \) are the divisors of \( a_n \), i.e., \( q = \pm1, \pm2 \).
        \item Possible values for \( p \) are the divisors of \( a_0 \), i.e., \( p = \pm \) divisors of \( 2(k - 1) \).
    \end{itemize}
    Since \( m \geq 3 > 0 \), we focus on positive rational roots. Thus, the possible rational solutions are
    \[
    m_0 = \text{divisors of } 2(k - 1), \quad \text{or} \quad m_0 = \tfrac{1}{2} \times \text{divisors of } 2(k - 1).
    \]
    
    \item \textbf{Odd \( k \geq 3 \):}
    \begin{itemize}
        \item The leading coefficient is \( a_n = 2 \).
        \item The constant term is \( a_0 = 0 \).
    \end{itemize}
    Since \( a_0 = 0 \), the rational root theorem cannot be directly applied. However, we observe that \( m = 0 \) is a root. Factoring out \( m \), we can write
    \[
    P_{\mathbb{R}}(m) = m \cdot Q_{\mathbb{R}}(m),
    \]
    where \( Q_{\mathbb{R}}(m) \) is a polynomial with integer coefficients. If \( Q_{\mathbb{R}}(m) \) has a non-zero constant term, we can apply the rational root theorem to \( Q_{\mathbb{R}}(m) \).
\end{itemize}

By applying the rational root theorem appropriately in each case, we can narrow down the possible rational values of \( m \) that could satisfy \( P_{\mathbb{R}}(m) = 0 \).
\end{proof}

This analysis allows us to systematically investigate potential rational solutions for \( m \) based on the parity of \( k \). The detailed examination of these cases will be provided in the subsequent sections.

\subsubsection{Polynomial Nature of \( Q_{\mathbb{R}}(m) = 0 \)}

We investigate whether \( Q_{\mathbb{R}}(m) = 0 \) constitutes a polynomial equation with integer coefficients \( a_i \in \mathbb{Z} \) and non-zero leading and constant terms \( a_n, a_0 \in \mathbb{Z}^* \).

\begin{lemma}\label{lemma:Qr_polynomial}
For odd integers \( k \geq 3 \), the equation \( Q_{\mathbb{R}}(m) = 0 \) is a polynomial equation with integer coefficients \( a_i \in \mathbb{Z} \) and non-zero leading and constant terms \( a_n, a_0 \in \mathbb{Z}^* \).
\end{lemma}

\begin{proof}
First, we express \( Q_{\mathbb{R}}(m) \) as
\begin{align*}
Q_{\mathbb{R}}(m) &= \frac{P_{\mathbb{R}}(m)}{m} \\
&\approx \frac{2}{m}(m - 1)^{k + 1} + \frac{k + 1}{m}(m - 1)^{k} - 2(k + 1)m^{k - 1} + \frac{k - 1}{m}.
\end{align*}

Previously, we showed that when \( k \) is odd, the constant term of \( P_{\mathbb{R}}(m) \) is \( a_0 = 0 \). Therefore, \( P_{\mathbb{R}}(m) \) has no constant term, and when we divide by \( m \), there are no terms involving \( \frac{1}{m} \) in \( Q_{\mathbb{R}}(m) \).

Moreover, the terms in \( P_{\mathbb{R}}(m) \) of degree higher than zero decrease in degree by one when forming \( Q_{\mathbb{R}}(m) \). As a result, \( Q_{\mathbb{R}}(m) \) is a polynomial with integer coefficients.

Additionally, the constant term of \( Q_{\mathbb{R}}(m) \) arises from the terms in \( P_{\mathbb{R}}(m) \) that are linear in \( m \). Since these terms are divided by \( m \), they become constants in \( Q_{\mathbb{R}}(m) \).

Thus, \( Q_{\mathbb{R}}(m) = 0 \) is indeed a polynomial equation with integer coefficients \( a_i \in \mathbb{Z} \) and non-zero leading and constant terms \( a_n, a_0 \in \mathbb{Z}^* \).
\end{proof} 

\subsubsection{Application of the Rational Root Theorem to \( Q_{\mathbb{R}}(m) = 0 \)}

We proceed to apply the rational root theorem to the polynomial equation \( Q_{\mathbb{R}}(m) = 0 \), which has integer coefficients and non-zero leading and constant terms.

\begin{lemma}\label{lemma:rational_root_Qr}
For odd integers \( k \geq 3 \), all rational roots \( m_0 = \frac{p}{q} \) of the equation \( Q_{\mathbb{R}}(m) = 0 \) (expressed in lowest terms) satisfy:
\begin{itemize}
    \item \( p \) divides the constant term \( a_0 \),
    \item \( q \) divides the leading coefficient \( a_n \).
\end{itemize}
Specifically, possible values for \( m_0 \) are the divisors of \( (k + 1)(k - 2) \) and half of those divisors.
\end{lemma}

\begin{proof}
First, we expand the terms \( (m - 1)^{k + 1} \) and \( (m - 1)^{k} \) using the binomial theorem:
\begin{align*}
(m - 1)^{k + 1} &= \sum_{l = 0}^{k + 1} \binom{k + 1}{l} (-1)^{k + 1 - l} m^{l}, \\
(m - 1)^{k} &= \sum_{l = 0}^{k} \binom{k}{l} (-1)^{k - l} m^{l}.
\end{align*}

Next, we identify the coefficient of \( m^{1} \) in \( P_{\mathbb{R}}(m) \). From the expansions, the \( m^{1} \) terms are:
\begin{align*}
2(m - 1)^{k + 1} &\rightarrow 2 \binom{k + 1}{1} (-1)^{k} m^{1} = 2(k + 1)(-1)^{k} m^{1}, \\
(k + 1)(m - 1)^{k} &\rightarrow (k + 1) \binom{k}{1} (-1)^{k - 1} m^{1} = (k + 1)k(-1)^{k - 1} m^{1}.
\end{align*}

Combining these, the coefficient \( a_1 \) of \( m^{1} \) in \( P_{\mathbb{R}}(m) \) is:
\[
a_1 = 2(k + 1)(-1)^{k} + (k + 1)k(-1)^{k - 1}.
\]

Simplifying \( a_1 \) for odd \( k \geq 3 \):
\[
(-1)^{k} = -1, \quad (-1)^{k - 1} = 1.
\]
Thus,
\begin{align*}
a_1 &= 2(k + 1)(-1) + (k + 1)k(1) \\
&= -2(k + 1) + (k + 1)k \\
&= (k + 1)(k - 2).
\end{align*}

For the polynomial \( Q_{\mathbb{R}}(m) \), the constant term \( a_0 \) comes from the linear term in \( P_{\mathbb{R}}(m) \) divided by \( m \), so \( a_0 = a_1 \).

The leading coefficient \( a_n \) of \( Q_{\mathbb{R}}(m) \) remains \( a_n = 2 \), as the highest degree term in \( P_{\mathbb{R}}(m) \) is \( 2(m - 1)^{k + 1} \), which upon division by \( m \) reduces the degree by one.

Applying the rational root theorem, possible rational roots \( m_0 = \frac{p}{q} \) satisfy:
\begin{itemize}
    \item \( p \) divides \( a_0 = (k + 1)(k - 2) \),
    \item \( q \) divides \( a_n = 2 \).
\end{itemize}

Since \( m \geq 3 > 0 \), we consider positive rational roots. Therefore, possible values for \( m_0 \) are the divisors of \( (k + 1)(k - 2) \) and half of those divisors, i.e.,
\[
m_0 \in \left\{ \text{divisors of } (k + 1)(k - 2),\ \tfrac{1}{2} \times \text{divisors of } (k + 1)(k - 2) \right\}.
\]

This concludes the application of the rational root theorem to \( Q_{\mathbb{R}}(m) = 0 \).
\end{proof}

\subsubsection{Summary of Rational Roots of \( P_{\mathbb{R}}(m) \)}
\label{sec:sum_rat_roots}

We summarize the possible rational roots of the polynomial \( P_{\mathbb{R}}(m) = 0 \) for integers \( k \geq 2 \) and \( m \geq 3 \).

\begin{lemma}\label{lemma:rational_roots_summary}
For \( k \geq 2 \), the possible rational roots \( m_0 \) of \( P_{\mathbb{R}}(m) = 0 \) are as follows:

\begin{itemize}
    \item \textbf{Even \( k \geq 2 \):}
    \[
    m_0 = \text{divisors of } 2(k - 1), \quad \text{and} \quad m_0 = \tfrac{1}{2} \times \text{divisors of } 2(k - 1).
    \]
    
    \item \textbf{Odd \( k \geq 3 \):}
    \[
    m_0 = \text{divisors of } (k + 1)(k - 2), \quad \text{and} \quad m_0 = \tfrac{1}{2} \times \text{divisors of } (k + 1)(k - 2).
    \]
\end{itemize}
\end{lemma}

\begin{proof}
This follows directly from the application of the rational root theorem to \( P_{\mathbb{R}}(m) = 0 \), as established in the previous subsections. The divisibility conditions on \( p \) and \( q \) determine the possible rational roots based on the constant term and leading coefficient of the polynomial.
\end{proof}

To satisfy the constraint \( m \geq 3 \), we further restrict \( k \) and identify the integer rational roots as:

\begin{itemize}
    \item \textbf{For even \( k \geq 4 \):}
    \[
    m_0 \in \{ k - 1,\ 2(k - 1) \}.
    \]
    
    \item \textbf{For odd \( k \geq 5 \):}
    \[
    m_0 \in \{ k - 2 \}.
    \]
    
    \item \textbf{For odd \( k \geq 3 \):}
    \[
    m_0 \in \{ k + 1,\ (k + 1)(k - 2) \}.
    \]
\end{itemize}

By constraining the rational roots to integers, we can refocus our analysis from the approximate equation \( P_{\mathbb{R}}(m) = 0 \) back to the original integer equation \( P(m) = 0 \).

\subsection{Analysis of \( P_{\mathbb{R}}(m) \) and \( Q_{\mathbb{R}}(m) \) at \( m = m_0 \)}

We analyze the sign of \( P_{\mathbb{R}}(m) \) at the candidate integer rational roots \( m_0 \) identified earlier. Recall that
\[
P_{\mathbb{R}}(m) \approx 2(m - 1)^{k + 1} + (k + 1)(m - 1)^{k} - 2(k + 1)m^{k} + (k - 1).
\]

\subsubsection{Case 1: Even \( k \geq 4 \)}

\begin{lemma}\label{lemma:Pmr_sign_even_k}
For even integers \( k \geq 4 \), we have \( P_{\mathbb{R}}(m_0) < 0 \) at \( m_0 = k - 1 \).
\end{lemma}

\begin{proof}
At \( m = k - 1 \), we have \( m - 1 = k - 2 \). Substituting into \( P_{\mathbb{R}}(m) \), we obtain
\[
P_{\mathbb{R}}(k - 1) \approx \underbrace{2(k - 2)^{k + 1}}_{A} + \underbrace{(k + 1)(k - 2)^{k}}_{B} \underbrace{-2(k + 1)(k - 1)^{k}}_{C} + \underbrace{(k - 1)}_{D}.
\]

Evaluating \( P_{\mathbb{R}}(k - 1) \) for specific values of \( k \):
\begin{align*}
k &= 4, & P_{\mathbb{R}}(3) &= -663, \\
k &= 6, & P_{\mathbb{R}}(5) &\approx -1.57 \times 10^{5}, \\
k &= 8, & P_{\mathbb{R}}(7) &\approx -6.85 \times 10^{7}, \\
k &= 10, & P_{\mathbb{R}}(9) &\approx -4.77 \times 10^{10}.
\end{align*}

Next, we consider the ratio \( R = \frac{|C|}{A + B} \):
\begin{align*}
R &= \frac{2(k + 1)(k - 1)^{k}}{2(k - 2)^{k + 1} + (k + 1)(k - 2)^{k}} \\
  &= \frac{2(k + 1)}{3(k - 1)} \left( \frac{k - 1}{k - 2} \right)^{k}.
\end{align*}

We analyze the limit \( \mathcal{L} \) of \( R \) as \( k \to \infty \). Observe that
\[
\frac{2(k + 1)}{3(k - 1)} = \frac{2 + \frac{2}{k}}{3 - \frac{3}{k}} \to \frac{2}{3} \quad \text{as} \quad k \to \infty.
\]
Also,
\begin{align*}
\ln \left( \left( \frac{k - 1}{k - 2} \right)^{k} \right) &= k \ln \left( 1 + \frac{1}{k - 2} \right) \\
&\approx \frac{k}{k - 2} - \frac{k}{2(k - 2)^{2}} \to 1 \quad \text{as} \quad k \to \infty.
\end{align*}
Thus,
\[
\left( \frac{k - 1}{k - 2} \right)^{k} \to e \quad \text{as} \quad k \to \infty.
\]
Therefore, the limit is
\[
\mathcal{L} = \lim_{k \to \infty} R = \frac{2e}{3}.
\]

To analyze the monotonicity of \( R \), we define \( L(k) = \ln R \):
\[
L(k) = \ln R = \underbrace{\ln \left( \frac{2(k + 1)}{3(k - 1)} \right)}_{A(k)} + \underbrace{k \ln \left( \frac{k - 1}{k - 2} \right)}_{B(k)}.
\]
Compute the derivatives \( A'(k) \) and \( B'(k) \):
\begin{align*}
A'(k) &= \frac{d}{dk} \ln \left( \frac{2(k + 1)}{3(k - 1)} \right) = \frac{-2}{(k + 1)(k - 1)}, \\
B'(k) &= \ln \left( \frac{k - 1}{k - 2} \right) + k \left( \frac{1}{k - 1} - \frac{1}{k - 2} \right) \\
&= \ln \left( \frac{k - 1}{k - 2} \right) - \frac{k}{(k - 1)(k - 2)}.
\end{align*}
Therefore, the derivative of \( L(k) \) is
\[
L'(k) = A'(k) + B'(k) = \ln \left( \frac{k - 1}{k - 2} \right) - \left[ \frac{2}{(k + 1)(k - 1)} + \frac{k}{(k - 1)(k - 2)} \right].
\]
Let
\[
C(k) = \frac{2}{(k + 1)(k - 1)} + \frac{k}{(k - 1)(k - 2)}.
\]
To show that \( L'(k) < 0 \) for all \( k \geq 4 \), we note that
\[
\ln \left( \frac{k - 1}{k - 2} \right) = \ln \left( 1 + \frac{1}{k - 2} \right) < \frac{1}{k - 2},
\]
since \( \ln(1 + x) < x \) for \( x > 0 \). Also,
\[
C(k) > \frac{1}{k - 2}, \quad \forall k \geq 4.
\]
Therefore,
\[
\ln \left( \frac{k - 1}{k - 2} \right) < \frac{1}{k - 2} < C(k),
\]
implying that
\[
L'(k) = \ln \left( \frac{k - 1}{k - 2} \right) - C(k) < 0.
\]
Since \( L'(k) < 0 \) for all \( k \geq 4 \), \( L(k) \) is decreasing on \( [4, \infty) \). Consequently, \( R = e^{L(k)} \) is also decreasing.

Given that \( R > 1 \) and \( |C| > A + B \), and considering \( D \) is positive but negligible compared to \( |C| \), we conclude that
\[
P_{\mathbb{R}}(k - 1) < 0 \quad \text{for all even } k \geq 4.
\]
\end{proof}

This confirms that for even \( k \geq 4 \), \( P_{\mathbb{R}}(m) \) is negative at \( m = m_0 = k - 1 \), indicating that these values do not yield a solution to the equation \( P_{\mathbb{R}}(m) = 0 \).

\subsubsection{Continuation of the Analysis for Even \( k \geq 4 \)}

We now consider the second candidate \( m_0 = 2(k - 1) \) for even \( k \geq 4 \).

\begin{lemma}\label{lemma:Pmr_sign_even_k_second_candidate}
For even integers \( k \geq 4 \), the sign of \( P_{\mathbb{R}}(m) \) at \( m = m_0 = 2(k - 1) \) changes depending on the value of \( k \).
\end{lemma}

\begin{proof}
At \( m_0 = 2(k - 1) \), we have \( m - 1 = 2k - 3 \). Substituting into \( P_{\mathbb{R}}(m) \), we obtain
\[
P_{\mathbb{R}}(2(k - 1)) \approx \underbrace{2(2k - 3)^{k + 1}}_{E} + \underbrace{(k + 1)(2k - 3)^{k}}_{F} \underbrace{-2(k + 1)[2(k - 1)]^{k}}_{G} + \underbrace{(k - 1)}_{H}.
\]

First, we combine \( E \) and \( F \):
\begin{align*}
E + F &= (2k - 3)^{k} \left[ 2(2k - 3) + (k + 1) \right] \\
      &= (2k - 3)^{k} \left[ 4k - 6 + k + 1 \right] \\
      &= (2k - 3)^{k} \left[ 5k - 5 \right] \\
      &= 5(k - 1)(2k - 3)^{k}.
\end{align*}

Next, we define the ratio \( R = \dfrac{|G|}{E + F} \):
\begin{align*}
R &= \frac{2(k + 1)[2(k - 1)]^{k}}{5(k - 1)(2k - 3)^{k}} \\
  &= \frac{2(k + 1)}{5(k - 1)} \left( \frac{2(k - 1)}{2k - 3} \right)^{k}.
\end{align*}

We compute numerical values of \( R \) for specific \( k \):
\begin{align*}
k &= 4, & R &\approx 1.382, \\
k &= 6, & R &\approx 1.054, \\
k &= 8, & R &\approx 0.930, \\
k &= 10, & R &\approx 0.866.
\end{align*}

We observe that for \( k = 4, 6 \), \( R > 1 \), and for \( k \geq 8 \), \( R < 1 \).

Now, we study the limit \( \mathcal{L} \) of \( R \) as \( k \rightarrow \infty \). First,
\[
\frac{2(k + 1)}{5(k - 1)} = \frac{2 + \frac{2}{k}}{5 - \frac{5}{k}} \rightarrow \frac{2}{5} \quad \text{as} \quad k \rightarrow \infty.
\]
Next,
\begin{align*}
\ln \left( \left( \frac{2(k - 1)}{2k - 3} \right)^{k} \right) &= k \ln \left( 1 + \frac{1}{2k - 3} \right) \\
&\approx \frac{k}{2k - 3} - \frac{k}{2(2k - 3)^{2}} \rightarrow \frac{1}{2} \quad \text{as} \quad k \rightarrow \infty.
\end{align*}
Therefore,
\[
\left( \frac{2(k - 1)}{2k - 3} \right)^{k} \rightarrow e^{1/2} = \sqrt{e}.
\]
Thus, the limit is
\[
\mathcal{L} = \lim_{k \rightarrow \infty} R = \frac{2\sqrt{e}}{5}.
\]

To confirm the monotonicity of \( R \) for \( k \geq 8 \), we consider \( L(k) = \ln R \):
\[
L(k) = \ln R = \underbrace{\ln \left( \frac{2(k + 1)}{5(k - 1)} \right)}_{A(k)} + \underbrace{k \ln \left( \frac{2(k - 1)}{2k - 3} \right)}_{B(k)}.
\]

Compute the derivatives \( A'(k) \) and \( B'(k) \):
\begin{align*}
A'(k) &= \frac{1}{k + 1} - \frac{1}{k - 1}, \\
B'(k) &= \ln \left( \frac{2(k - 1)}{2k - 3} \right) + k \left( \frac{1}{k - 1} - \frac{2}{2k - 3} \right).
\end{align*}

Combining, the derivative \( L'(k) \) is
\begin{align*}
L'(k) &= A'(k) + B'(k) \\
&= \left( \frac{1}{k + 1} - \frac{1}{k - 1} \right) + \ln \left( \frac{2(k - 1)}{2k - 3} \right) + k \left( \frac{1}{k - 1} - \frac{2}{2k - 3} \right) \\
&= \frac{1}{k + 1} + 1 + \ln \left( \frac{2(k - 1)}{2k - 3} \right) - \frac{2k}{2k - 3} \\
&= \frac{1}{k + 1} + \ln \left( \frac{2(k - 1)}{2k - 3} \right) - \frac{3}{2k - 3}.
\end{align*}

Simplifying:
\begin{align*}
L'(k) &= \frac{1}{k + 1} + \ln \left( 1 + \frac{1}{2k - 3} \right) - \frac{3}{2k - 3}.
\end{align*}

For \( k \geq 8 \), we approximate the logarithm using \( \ln(1 + x) \approx x - \frac{x^{2}}{2} \), with \( x = \frac{1}{2k - 3} \):
\begin{align*}
\ln \left( 1 + \frac{1}{2k - 3} \right) &\approx \frac{1}{2k - 3} - \frac{1}{2(2k - 3)^{2}}.
\end{align*}

We approximate \( L'(k) \):
\begin{align*}
L'(k) &\approx \frac{1}{k + 1} + \left( \frac{1}{2k - 3} - \frac{1}{2(2k - 3)^{2}} \right) - \frac{3}{2k - 3} \\
&\approx \frac{1}{k + 1} - \frac{2}{2k - 3} - \frac{1}{2(2k - 3)^{2}} \\
&\approx  - \frac{5}{(k + 1)(2k - 3)} - \frac{1}{2(2k - 3)^{2}}.
\end{align*}

Thus, we find that \( L'(k) < 0 \) for \( k \geq 8 \). Therefore, \( L(k) \) is decreasing on \( [8, \infty) \), implying that \( R = e^{L(k)} \) is also decreasing.

Given that \( R \) decreases from approximately \( 0.930 \) at \( k = 8 \) to \( \frac{2\sqrt{e}}{5} \approx 0.659 \) as \( k \rightarrow \infty \), and \( R > 0 \), we can conclude:

- For \( k = 4, 6 \), \( R > 1 \), so \( |G| > E + F \), and since \( E + F > 0 \), \( G < 0 \), \( H > 0 \) but \( H \ll |G| \), we have \( P_{\mathbb{R}}(m_0) < 0 \).

- For \( k \geq 8 \), \( R < 1 \), so \( |G| < E + F \), and with \( E + F > 0 \), \( G < 0 \), \( H > 0 \), we have \( P_{\mathbb{R}}(m_0) > 0 \).

Therefore, the sign of \( P_{\mathbb{R}}(m_0) \) at \( m_0 = 2(k - 1) \) changes depending on \( k \), being negative for \( k = 4, 6 \) and positive for \( k \geq 8 \).
\end{proof}

\subsubsection{Implications for \( Q_{\mathbb{R}}(m) \)}

Recall that \( Q_{\mathbb{R}}(m) = \dfrac{P_{\mathbb{R}}(m)}{m} \). By evaluating the sign of \( P_{\mathbb{R}}(m) \) at the candidate integer rational roots of \( Q_{\mathbb{R}}(m) \), we can deduce the sign of \( Q_{\mathbb{R}}(m) \) at \( m = m_0 \) for odd \( k \geq 3 \) and \( m \geq 3 \).

This analysis indicates that the candidate integer values of \( m \) do not satisfy \( P_{\mathbb{R}}(m) = 0 \) or \( Q_{\mathbb{R}}(m) = 0 \) for the respective ranges of \( k \), supporting the conjecture that no solutions exist for \( k \geq 2 \) and \( m \geq 3 \).

\subsubsection{Case 2: Odd \( k \geq 5 \)}

We now analyze the sign of \( P_{\mathbb{R}}(m) \) at \( m_0 = k - 2 \) for odd integers \( k \geq 5 \).

\begin{lemma}\label{lemma:Pmr_sign_odd_k}
For odd integers \( k \geq 5 \), we have \( P_{\mathbb{R}}(m_0) < 0 \) at \( m_0 = k - 2 \).
\end{lemma}

\begin{proof}
At \( m = k - 2 \), we have \( m - 1 = k - 3 \). Substituting into \( P_{\mathbb{R}}(m) \), we obtain
\[
P_{\mathbb{R}}(k - 2) \approx \underbrace{2(k - 3)^{k + 1}}_{A} + \underbrace{(k + 1)(k - 3)^{k}}_{B} \underbrace{-2(k + 1)(k - 2)^{k}}_{C} + \underbrace{(k - 1)}_{D}.
\]

First, combine \( A \) and \( B \):
\begin{align*}
A + B &= 2(k - 3)^{k + 1} + (k + 1)(k - 3)^{k} \\
      &= (k - 3)^{k} \left[ 2(k - 3) + (k + 1) \right] \\
      &= (k - 3)^{k} \left( 2k - 6 + k + 1 \right) \\
      &= (3k - 5)(k - 3)^{k}.
\end{align*}

Next, define the ratio \( R = \dfrac{|C|}{A + B} \):
\begin{align*}
R &= \frac{2(k + 1)(k - 2)^{k}}{(3k - 5)(k - 3)^{k}} \\
  &= \frac{2(k + 1)}{3k - 5} \left( \frac{k - 2}{k - 3} \right)^{k}.
\end{align*}

Compute numerical values of \( R \) for specific \( k \):
\begin{align*}
k &= 5, & R &\approx 9.112, \\
k &= 7, & R &\approx 4.768, \\
k &= 9, & R &\approx 3.640, \\
k &= 11, & R &\approx 3.131.
\end{align*}

We observe that for \( k \in \{5, 7, 9, 11\} \), \( R > 1 \).

We study the limit \( \mathcal{L} \) of \( R \) as \( k \rightarrow \infty \). First,
\[
\frac{2(k + 1)}{3k - 5} = \frac{2 + \frac{2}{k}}{3 - \frac{5}{k}} \rightarrow \frac{2}{3} \quad \text{as} \quad k \rightarrow \infty.
\]
Next,
\begin{align*}
\ln \left( \left( \frac{k - 2}{k - 3} \right)^{k} \right) &= k \ln \left( 1 + \frac{1}{k - 3} \right) \\
&\approx \frac{k}{k - 3} - \frac{k}{2(k - 3)^{2}} \rightarrow 1 \quad \text{as} \quad k \rightarrow \infty.
\end{align*}
Therefore,
\[
\left( \frac{k - 2}{k - 3} \right)^{k} \rightarrow e \quad \text{as} \quad k \rightarrow \infty.
\]

Thus, the limit is
\[
\mathcal{L} = \lim_{k \rightarrow \infty} R = \frac{2e}{3}.
\]

To confirm the monotonicity of \( R \) for \( k \geq 5 \), we define \( L(k) = \ln R \):
\[
L(k) = \ln R = \underbrace{\ln \left( \frac{2(k + 1)}{3k - 5} \right)}_{A(k)} + \underbrace{k \ln \left( \frac{k - 2}{k - 3} \right)}_{B(k)}.
\]

Compute the derivatives \( A'(k) \) and \( B'(k) \):
\begin{align*}
A'(k) &= \frac{2}{k + 1} - \frac{3}{3k - 5} \\
      &= - \frac{8}{3(k+1)(k - 5/3)}, \\
B'(k) &= \ln \left( \frac{k - 2}{k - 3} \right) + k \left( \frac{1}{k - 2} - \frac{1}{k - 3} \right).
\end{align*}

Simplify \( B'(k) \):
\[
B'(k) = \ln \left( 1 + \frac{1}{k - 3} \right) - \frac{k}{(k - 2)(k - 3)}.
\]

Therefore, the derivative of \( L(k) \) is
\[
L'(k) = - \frac{8}{3(k+1)(k - 5/3)} + \ln \left( 1 + \frac{1}{k - 3} \right) - \frac{k}{(k - 2)(k - 3)}.
\]

To determine the sign of \( L'(k) \) for \( k \geq 5 \), we approximate the logarithm using \( \ln(1 + x) \approx x - \frac{x^{2}}{2} \) for small \( x \):
\[
\ln \left( 1 + \frac{1}{k - 3} \right) \approx \frac{1}{k - 3} - \frac{1}{2(k - 3)^{2}}.
\]

Substituting back, we find that
\begin{align*}
L'(k) &= - \frac{8}{3(k+1)(k - 5/3)} + \frac{1}{k - 3} - \frac{1}{2(k - 3)^{2}} - \frac{k}{(k - 2)(k - 3)} \\
&= - \frac{8}{3(k+1)(k - 5/3)} - \frac{1}{2(k - 3)^{2}} - \frac{2}{(k - 2)(k - 3)}.
\end{align*}

Thus, we find that \( L'(k) < 0 \) for \( k \geq 5 \). Therefore, \( L(k) \) is decreasing on \( [5, \infty) \), implying that \( R = e^{L(k)} \) is also decreasing.

Given that \( R > 1 \) and \( |C| > A + B \), and considering \( D \) is positive but negligible compared to \( |C| \), we conclude that
\[
P_{\mathbb{R}}(k - 2) < 0 \quad \text{for all odd } k \geq 5.
\]
\end{proof}

This confirms that for odd \( k \geq 5 \), \( P_{\mathbb{R}}(m) \) is negative at \( m = m_0 = k - 2 \), indicating that these values do not yield a solution to the equation \( P_{\mathbb{R}}(m) = 0 \).
 
\subsubsection{Case 3: Odd \( k \geq 3 \)}

We analyze the sign of \( P_{\mathbb{R}}(m) \) at \( m_0 = k + 1 \) for odd integers \( k \geq 3 \).

\begin{lemma}\label{lemma:Pmr_sign_odd_k_case3}
For odd integers \( k \geq 3 \), we have \( P_{\mathbb{R}}(m_0) < 0 \) at \( m_0 = k + 1 \).
\end{lemma}

\begin{proof}
At \( m = k + 1 \), we have \( m - 1 = k \). Substituting into \( P_{\mathbb{R}}(m) \), we obtain
\[
P_{\mathbb{R}}(k + 1) \approx \underbrace{2k^{k + 1}}_{E} + \underbrace{(k + 1)k^{k}}_{F} \underbrace{-2(k + 1)(k + 1)^{k}}_{G} + \underbrace{(k - 1)}_{H}.
\]

First, combine \( E \) and \( F \):
\begin{align*}
E + F &= 2k^{k + 1} + (k + 1)k^{k} \\
      &= k^{k} \left[ 2k + (k + 1) \right] \\
      &= (3k + 1)k^{k}.
\end{align*}

Next, define the ratio \( R = \dfrac{|G|}{E + F} \):
\begin{align*}
R &= \frac{2(k + 1)(k + 1)^{k}}{(3k + 1)k^{k}} \\
  &= \frac{2(k + 1)}{3k + 1} \left( \frac{k + 1}{k} \right)^{k}.
\end{align*}

Compute numerical values of \( R \) for specific \( k \):
\begin{align*}
k &= 3, & R &\approx 1.896, \\
k &= 5, & R &\approx 1.866, \\
k &= 7, & R &\approx 1.852, \\
k &= 9, & R &\approx 1.844.
\end{align*}

We observe that for \( k \in \{3, 5, 7, 9\} \), \( R > 1 \).

We study the limit \( \mathcal{L} \) of \( R \) as \( k \rightarrow \infty \). First,
\[
\frac{2(k + 1)}{3k + 1} = \frac{2 + \frac{2}{k}}{3 + \frac{1}{k}} \rightarrow \frac{2}{3} \quad \text{as} \quad k \rightarrow \infty.
\]
Next,
\begin{align*}
\ln \left( \left( \frac{k + 1}{k} \right)^{k} \right) &= k \ln \left( 1 + \frac{1}{k} \right) \\
&= k \left( \frac{1}{k} - \frac{1}{2k^{2}} + \frac{1}{3k^{3}} - \dots \right) \\
&= 1 - \frac{1}{2k} + \frac{1}{3k^{2}} - \dots \rightarrow 1 \quad \text{as} \quad k \rightarrow \infty.
\end{align*}
Thus,
\[
\left( \frac{k + 1}{k} \right)^{k} \rightarrow e \quad \text{as} \quad k \rightarrow \infty.
\]
Therefore, the limit is
\[
\mathcal{L} = \lim_{k \rightarrow \infty} R = \frac{2e}{3}.
\]

To confirm the monotonicity of \( R \) for \( k \geq 3 \), we define \( L(k) = \ln R \):
\[
L(k) = \ln R = \underbrace{\ln \left( \frac{2(k + 1)}{3k + 1} \right)}_{A(k)} + \underbrace{k \ln \left( \frac{k + 1}{k} \right)}_{B(k)}.
\]

Compute the derivatives \( A'(k) \) and \( B'(k) \):
\begin{align*}
A'(k) &= \frac{1}{k + 1} - \frac{3}{3k + 1} \\
      &= \frac{3k + 1 - 3(k + 1)}{(k + 1)(3k + 1)} = \frac{-2}{(k + 1)(3k + 1)}, \\
B'(k) &= \ln \left( \frac{k + 1}{k} \right) + k \left( \frac{1}{k + 1} - \frac{1}{k} \right) \\
      &= \ln \left( 1 + \frac{1}{k} \right) - \frac{1}{k + 1}.
\end{align*}

Therefore, the derivative of \( L(k) \) is
\[
L'(k) = A'(k) + B'(k) = \frac{-2}{(k + 1)(3k + 1)} + \ln \left( 1 + \frac{1}{k} \right) - \frac{1}{k + 1}.
\]

To determine the sign of \( L'(k) \) for \( k \geq 3 \), we use the approximation \( \ln(1 + x) \approx x - \frac{x^{2}}{2} \) for small \( x \):
\[
\ln \left( 1 + \frac{1}{k} \right) \approx \frac{1}{k} - \frac{1}{2k^{2}}.
\]

Substituting back, we get
\begin{align*}
L'(k) &\approx \frac{-2}{(k + 1)(3k + 1)} + \left( \frac{1}{k} - \frac{1}{2k^{2}} \right) - \frac{1}{k + 1} \\
      &= \left( \frac{1}{k} - \frac{1}{k + 1} - \frac{2}{(k + 1)(3k + 1)} \right) - \frac{1}{2k^{2}}.
\end{align*}

Since for \( k \geq 3 \), the term inside the parentheses is negative (because \( \frac{1}{k} < \frac{1}{k + 1} \)), and the last term \( -\frac{1}{2k^{2}} \) is negative, we have \( L'(k) < 0 \).

Therefore, \( L(k) \) is decreasing on \( [3, \infty) \). Consequently, \( R = e^{L(k)} \) is also decreasing.

Given that \( R > 1 \) and \( |G| > E + F \), and considering \( H \) is positive but negligible compared to \( |G| \), we conclude that
\[
P_{\mathbb{R}}(k + 1) < 0 \quad \text{for all odd } k \geq 3.
\]
\end{proof}

This confirms that for odd \( k \geq 3 \), \( P_{\mathbb{R}}(m) \) is negative at \( m = m_0 = k + 1 \), indicating that these values do not yield a solution to the equation \( P_{\mathbb{R}}(m) = 0 \).

\subsubsection{Case 4: Odd \( k \geq 3 \)}

We analyze the sign of \( P_{\mathbb{R}}(m) \) at \( m_0 = (k + 1)(k - 2) \) for odd integers \( k \geq 3 \).

\begin{lemma}\label{lemma:Pmr_sign_odd_k_case4}
For odd integers \( k \geq 3 \), the sign of \( P_{\mathbb{R}}(m) \) at \( m = m_0 = (k + 1)(k - 2) \) is negative for \( k = 3 \) and positive for \( k \geq 5 \).
\end{lemma}

\begin{proof}
At \( m = (k + 1)(k - 2) \), we have \( m - 1 = k^{2} - k - 3 \). Substituting into \( P_{\mathbb{R}}(m) \), we obtain
\begin{align*}
P_{\mathbb{R}}((k + 1)(k - 2)) \approx &\underbrace{2(k^{2} - k - 3)^{k + 1}}_{I} + \underbrace{(k + 1)(k^{2} - k - 3)^{k}}_{J} \\
&\underbrace{-2(k + 1)[(k + 1)(k - 2)]^{k}}_{K} + \underbrace{(k - 1)}_{L}.
\end{align*}

First, combine \( I \) and \( J \):
\begin{align*}
I + J &= 2(k^{2} - k - 3)^{k + 1} + (k + 1)(k^{2} - k - 3)^{k} \\
      &= (k^{2} - k - 3)^{k} \left[ 2(k^{2} - k - 3) + (k + 1) \right] \\
      &= (k^{2} - k - 3)^{k} \left( 2k^{2} - 2k -6 + k +1 \right) \\
      &= (2k^{2} - k -5)(k^{2} - k - 3)^{k}.
\end{align*}

Next, define the ratio \( R = \dfrac{|K|}{I + J} \):
\begin{align*}
R &= \frac{2(k + 1)[(k + 1)(k - 2)]^{k}}{(2k^{2} - k - 5)(k^{2} - k - 3)^{k}} \\
  &= \frac{2(k + 1)}{2k^{2} - k - 5} \left( \frac{(k + 1)(k - 2)}{k^{2} - k - 3} \right)^{k}.
\end{align*}

Compute numerical values of \( R \) for specific \( k \):
\begin{align*}
k &= 3, & R &\approx 1.896, \\
k &= 5, & R &\approx 0.399, \\
k &= 7, & R &\approx 0.222, \\
k &= 9, & R &\approx 0.154.
\end{align*}

We observe that \( R > 1 \) for \( k = 3 \) and \( R < 1 \) for \( k \geq 5 \).

We study the limit \( \mathcal{L} \) of \( R \) as \( k \rightarrow \infty \). First,
\[
\frac{2(k + 1)}{2k^{2} - k - 5} = \frac{2}{2k - \frac{1}{k} - \frac{5}{k^{2}}} \rightarrow 0^{+} \quad \text{as} \quad k \rightarrow \infty.
\]
Next,
\begin{align*}
\ln \left( \left( \frac{(k + 1)(k - 2)}{k^{2} - k - 3} \right)^{k} \right) &= k \ln \left( 1 + \frac{(k + 1)(k - 2) - (k^{2} - k - 3)}{k^{2} - k - 3} \right) \\
&= k \ln \left( 1 + \frac{k^{2} - k - 2 - k^{2} + k + 3}{k^{2} - k - 3} \right) \\
&= k \ln \left( 1 + \frac{1}{k^{2} - k - 3} \right).
\end{align*}

For large \( k \), \( k^{2} - k - 3 \approx k^{2} \), so
\[
\ln \left( \left( \frac{(k + 1)(k - 2)}{k^{2} - k - 3} \right)^{k} \right) \approx k \left( \frac{1}{k^{2}} - \frac{1}{2k^{4}} \right) \approx \frac{1}{k} \rightarrow 0^{+} \quad \text{as} \quad k \rightarrow \infty.
\]

Therefore,
\[
\left( \frac{(k + 1)(k - 2)}{k^{2} - k - 3} \right)^{k} \rightarrow 1^{+} \quad \text{as} \quad k \rightarrow \infty.
\]

Thus, the limit is
\[
\mathcal{L} = \lim_{k \rightarrow \infty} R = 0^{+}.
\]

To confirm the monotonicity of \( R \) for \( k \geq 5 \), we define \( L(k) = \ln R \):
\[
L(k) = \ln R = \underbrace{\ln \left( \frac{2(k + 1)}{2k^{2} - k - 5} \right)}_{A(k)} + \underbrace{k \ln \left( \frac{(k + 1)(k - 2)}{k^{2} - k - 3} \right)}_{B(k)}.
\]

Compute the derivatives \( A'(k) \) and \( B'(k) \):

\begin{align*}
    A'(k) &= \frac{1}{k+1} - \frac{4k-1}{2k^2-k-5} \\
          &= \frac{-2(k^2+2k+2)}{(k+1)(2k^2-k-5)} \\
    B'(k) &= \ln \left( \frac{(k+1)(k-2)}{k^2 -k-3} \right) + k \left( \frac{d}{dk} \left[ \ln \left( \frac{(k+1)(k-2)}{k^2-k-3} \right) \right] \right) \\
          &= \ln \left( \frac{(k+1)(k-2)}{k^2 -k-3} \right) + k(2k-1) \left[\frac{1}{(k+1)(k-2)} - \frac{1}{k^2-k-3} \right] \\
          &=\ln \left( \frac{(k+1)(k-2)}{k^2 -k-3} \right) + k(2k-1) \left[\frac{1}{(k+1)(k-2)} - \frac{1}{[(k+1)(k-2)-1]} \right]
\end{align*}

Therefore, the derivative of \(L(k)\), \(L'(k)\), is given by: 

\begin{align*}
L'(k) &= \frac{-2(k^2+2k+2)}{(k+1)(2k^2-k-5)} + \ln \left( \frac{(k+1)(k-2)}{k^2 -k-3} \right) \\
      &+ k(2k-1) \left[\frac{1}{(k+1)(k-2)} - \frac{1}{[(k+1)(k-2)-1]} \right]
\end{align*}

Now, let's determine the sign of \( L'(k) \) for \( k \in [5, \infty) \). Using the first two terms of the Taylor series expansion for \( \ln(1 + x) \approx x - \frac{x^2}{2} \), where \( |x| < 1 \), we have: 

\begin{align*}
L'(k) &\approx \frac{-2(k^2+2k+2)}{(k+1)(2k^2-k-5)} + \left( \frac{1}{k^2 -k-3} - \frac{1}{2(k^2 -k-3)^2} \right) \\
      &+ k(2k-1) \left[\frac{1}{(k+1)(k-2)} - \frac{1}{[(k+1)(k-2)-1]} \right]
\end{align*}

Since for \( k \geq 5 \), \( (k + 1) \), \( (2k^2-k-5) \), \( (k^2+2k+2) \), \( (k^2 -k-3) \), and \( k(2k-1) \) are all positives, and that \( (k+1)(k-2)-1 < (k+1)(k-2) \), it comes that \( L'(k) < 0\) for \(k \in [5,\infty)\). Therefore, we find that \( L'(k) < 0 \) for \( k \geq 5 \), indicating that \( L(k) \) is decreasing, and thus \( R \) is decreasing.

Given that \( R \) decreases from approximately \( 0.399 \) at \( k = 5 \) to \( 0^{+} \) as \( k \rightarrow \infty \), and \( R > 0 \), we can conclude:

- For \( k = 3 \), \( R > 1 \), so \( |K| > I + J \), and since \( I + J > 0 \), \( K < 0 \), \( L > 0 \) but negligible compared to \( |K| \), we have \( P_{\mathbb{R}}(m_0) < 0 \).

- For \( k \geq 5 \), \( R < 1 \), so \( |K| < I + J \), and with \( I + J > 0 \), \( K < 0 \), \( L > 0 \), we have \( P_{\mathbb{R}}(m_0) > 0 \).

Therefore, the sign of \( P_{\mathbb{R}}(m) \) at \( m = m_0 = (k + 1)(k - 2) \) is negative for \( k = 3 \) and positive for all odd \( k \geq 5 \).

\end{proof}

\subsubsection{Summary of the Sign of \( P_{\mathbb{R}}(m) \) at Candidate Rational Roots}

We now summarize the results of our analysis regarding the sign of \( P_{\mathbb{R}}(m) \) at the expected integer rational roots \( m_0 \).

\begin{lemma}\label{lemma:Pmr_sign_summary}
For \( k \geq 2 \) and \( m \geq 3 \), the polynomial \( P_{\mathbb{R}}(m) \) does not vanish at any of the candidate integer rational roots \( m_0 \) identified earlier. Specifically:

\begin{itemize}
    \item \textbf{Case 1: Even \( k \geq 4 \)}
        \begin{align*}
            &\text{At } m_0 = k - 1,\quad \forall\, k \geq 4,\quad P_{\mathbb{R}}(m_0) < 0. \\
            &\text{At } m_0 = 2(k - 1),\quad \text{for } k \in \{4, 6\},\quad P_{\mathbb{R}}(m_0) < 0. \\
            &\text{At } m_0 = 2(k - 1),\quad \forall\, k \geq 8,\quad P_{\mathbb{R}}(m_0) > 0.
        \end{align*}
        
    \item \textbf{Case 2: Odd \( k \geq 5 \)}
        \begin{align*}
            &\text{At } m_0 = k - 2,\quad \forall\, k \geq 5,\quad P_{\mathbb{R}}(m_0) < 0.
        \end{align*}
        
    \item \textbf{Case 3: Odd \( k \geq 3 \)}
        \begin{align*}
            &\text{At } m_0 = k + 1,\quad \forall\, k \geq 3,\quad P_{\mathbb{R}}(m_0) < 0. \\
            &\text{At } m_0 = (k + 1)(k - 2),\quad \text{for } k = 3,\quad P_{\mathbb{R}}(m_0) < 0. \\
            &\text{At } m_0 = (k + 1)(k - 2),\quad \forall\, k \geq 5,\quad P_{\mathbb{R}}(m_0) > 0.
        \end{align*}
\end{itemize}
\end{lemma}

\begin{proof}
This summary consolidates the results obtained in the previous analyses. In each case, we evaluated \( P_{\mathbb{R}}(m) \) at the specified values of \( m_0 \) and determined the sign of \( P_{\mathbb{R}}(m_0) \). The detailed calculations are provided in the respective subsections.

From our findings:

\begin{itemize}
    \item For even \( k \geq 4 \):
    \begin{itemize}
        \item At \( m_0 = k - 1 \), \( P_{\mathbb{R}}(m_0) < 0 \) for all \( k \geq 4 \).
        \item At \( m_0 = 2(k - 1) \):
        \begin{itemize}
            \item \( P_{\mathbb{R}}(m_0) < 0 \) for \( k = 4, 6 \).
            \item \( P_{\mathbb{R}}(m_0) > 0 \) for all \( k \geq 8 \).
        \end{itemize}
    \end{itemize}
    
    \item For odd \( k \geq 5 \):
    \begin{itemize}
        \item At \( m_0 = k - 2 \), \( P_{\mathbb{R}}(m_0) < 0 \) for all \( k \geq 5 \).
    \end{itemize}
    
    \item For odd \( k \geq 3 \):
    \begin{itemize}
        \item At \( m_0 = k + 1 \), \( P_{\mathbb{R}}(m_0) < 0 \) for all \( k \geq 3 \).
        \item At \( m_0 = (k + 1)(k - 2) \):
        \begin{itemize}
            \item \( P_{\mathbb{R}}(m_0) < 0 \) for \( k = 3 \).
            \item \( P_{\mathbb{R}}(m_0) > 0 \) for all \( k \geq 5 \).
        \end{itemize}
    \end{itemize}
\end{itemize}

Since \( P_{\mathbb{R}}(m_0) \neq 0 \) at all the candidate values of \( m_0 \), none of these values are roots of \( P_{\mathbb{R}}(m) \).

Furthermore, for odd \( k \geq 3 \), the same conclusion applies to \( Q_{\mathbb{R}}(m) \), as \( P_{\mathbb{R}}(m) = m Q_{\mathbb{R}}(m) \) and \( m \geq 3 \).

Therefore, \( P_{\mathbb{R}}(m) \) does not vanish for any \( k \geq 2 \) and \( m \geq 3 \) at the candidate values of \( m_0 \).
\end{proof}

Consequently, the integer values \( m_0 \) for:

\begin{itemize}
    \item Even \( k \geq 4 \): \( m_0 \in \{ k - 1,\ 2(k - 1) \} \),
    \item Odd \( k \geq 5 \): \( m_0 = k - 2 \),
    \item Odd \( k \geq 3 \): \( m_0 \in \{ k + 1,\ (k + 1)(k - 2) \} \),
\end{itemize}

are not roots of \( P_{\mathbb{R}}(m) \) or \( Q_{\mathbb{R}}(m) \). This implies that \( P_{\mathbb{R}}(m) \) is non-zero for all \( k \geq 2 \) and \( m \geq 3 \).

Therefore, no positive integers \( m \geq 3 \) satisfy \( S_{\mathbb{R}}(m - 1, k) = m^{k} \) for \( k \geq 2 \), which confirms the validity of Conjecture~\ref{conj:main} in these cases.

\subsection{Graphical Illustration of \( P_{\mathbb{R}}(m) \) Behavior}

To complement our analytical findings, we present graphical illustrations of \( P_{\mathbb{R}}(m) \) as a function of \( k \) for the various rational roots \( m_{0} \). Fig.~\ref{fig:P_m_plots} provides visual confirmation of the behavior of \( P_{\mathbb{R}}(m) \) in the cases analyzed and highlight the absence of integer rational roots \( m_{0} \) for \( P_{\mathbb{R}}(m) \), for \( k \geq 2 \) and \( m \geq 3 \).

\begin{figure}[p]
    \centering
    \includegraphics[angle=90, height=0.9\textheight, width=0.9\textwidth]{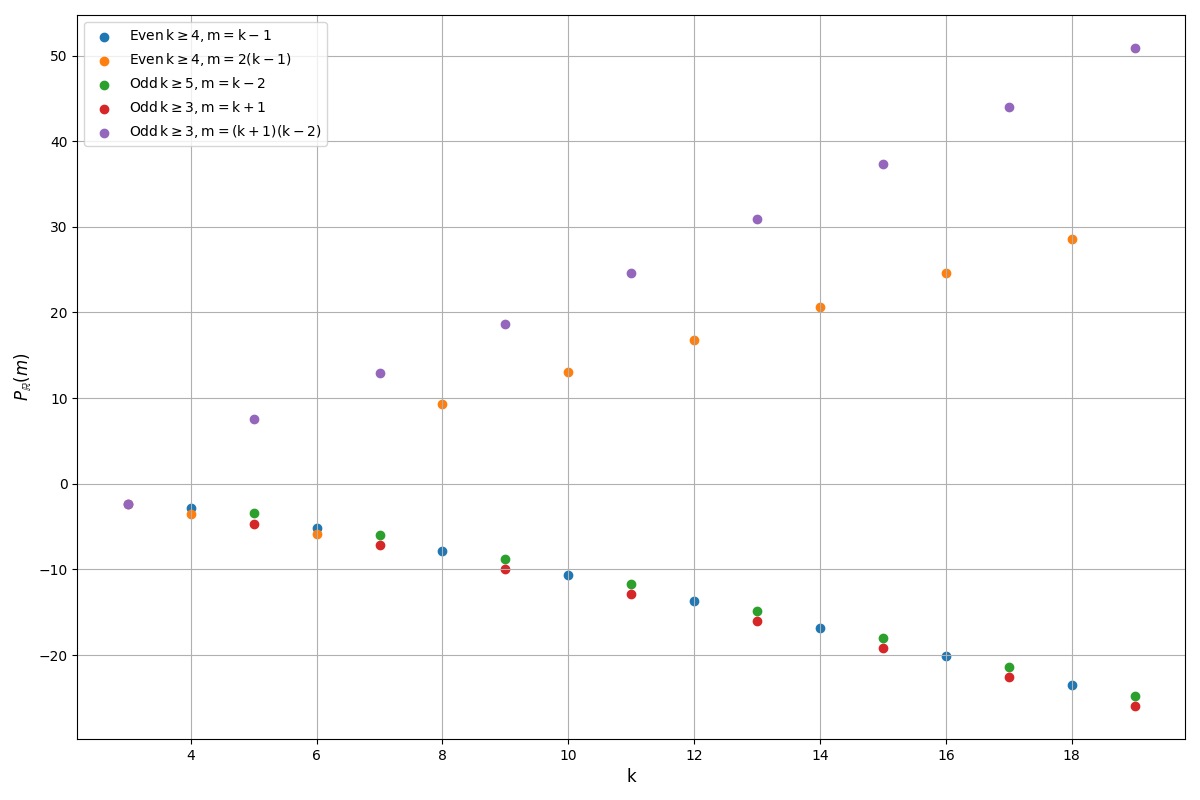}
    \caption{Plot of \( P_{\mathbb{R}}(m) \) as a function of \( k \) for the various rational roots \( m_{0} \) described in~\ref{sec:sum_rat_roots}}
    \label{fig:P_m_plots}
\end{figure}

\section{Extended Analysis Using the Full Euler-MacLaurin Expansion}

In this section, we provide a comprehensive and rigorous analysis of the sum \( S(m - 1, k) = \sum_{n=1}^{m - 1} n^{k} \) using the full Euler-MacLaurin expansion. By including all possible correction terms, we obtain an exact expression for \( S(m - 1, k) \) without any approximation errors. We then analyze the resulting polynomial \( \widetilde{P}(m) \) to demonstrate that no new integer rational roots are introduced as \( k \) increases.

\subsection{Applying the Euler-MacLaurin Formula to \( S(m - 1, k) \)}

The Euler-MacLaurin formula relates finite sums to integrals, incorporating correction terms involving higher derivatives and Bernoulli numbers~\cite{knopp1990theory}:

\begin{equation}\label{eq:EML_general}
\sum_{n=a}^{b} f(n) = \int_{a}^{b} f(x) \, dx + \frac{f(a) + f(b)}{2} + \sum_{r=1}^{p} \frac{B_{2r}}{(2r)!} \left( f^{(2r - 1)}(b) - f^{(2r - 1)}(a) \right) + R_p,
\end{equation}

where:
\begin{itemize}
    \item \( B_{2r} \) are Bernoulli numbers,
    \item \( f^{(n)}(x) \) denotes the \( n \)-th derivative of \( f(x) \),
    \item \( R_p \) is the remainder term after \( p \) correction terms,
    \item \( a \) and \( b \) are the lower and upper limits of the sum.
\end{itemize}

For \( f(x) = x^{k} \), the \( n \)-th derivative is:

\begin{equation}\label{eq:derivative}
f^{(n)}(x) = \begin{cases}
k(k - 1)\cdots(k - n + 1) x^{k - n} = k^{(n)} x^{k - n}, & \text{if } n \leq k, \\
0, & \text{if } n > k,
\end{cases}
\end{equation}

where \( k^{(n)} = k(k - 1)\cdots(k - n + 1) \) is the falling factorial.

Since \( f^{(n)}(x) = 0 \) for \( n > k \), the sum over correction terms in the Euler-MacLaurin formula terminates at \( r = \left\lfloor \dfrac{k}{2} \right\rfloor \), resulting in a finite number of correction terms and an exact expression for \( S(m - 1, k) \) without any remainder term (\( R_p = 0 \)):

\begin{equation}\label{eq:EML_exact}
S(m - 1, k) = \int_{1}^{m - 1} x^{k} \, dx + \frac{1^{k} + (m - 1)^{k}}{2} + \sum_{r=1}^{p_{\text{max}}} \frac{B_{2r}}{(2r)!} \left( f^{(2r - 1)}(m - 1) - f^{(2r - 1)}(1) \right),
\end{equation}

where \( p_{\text{max}} = \left\lfloor \dfrac{k}{2} \right\rfloor \).

\subsection{Formulating the Exact Polynomial \( P_{\mathbb{R}}(m) \)}

We define \( P_{\mathbb{R}}(m) \) as:

\begin{equation}\label{eq:P_R_def}
P_{\mathbb{R}}(m) = S(m - 1, k) - m^{k}.
\end{equation}

Our goal is to express \( P_{\mathbb{R}}(m) \) as a polynomial with integer coefficients. Using the exact expression of \( S(m - 1, k) \), we expand \( P_{\mathbb{R}}(m) \):

\begin{align}\label{eq:P_R_expanded}
P_{\mathbb{R}}(m) &= \frac{(m - 1)^{k + 1} - 1}{k + 1} + \frac{(m - 1)^{k} + 1}{2} - m^{k} + \sum_{r=1}^{p_{\text{max}}} T_r,
\end{align}

where \( T_r = \frac{B_{2r}}{(2r)!} k^{(2r - 1)} \left( (m - 1)^{k - (2r - 1)} - 1 \right) \).

We collect all constant terms into a single constant \( C \):

\begin{equation}\label{eq:C_def}
C = -\frac{1}{k + 1} + \frac{1}{2} - \sum_{r=1}^{p_{\text{max}}} \frac{B_{2r}}{(2r)!} k^{(2r - 1)}.
\end{equation}

Thus, \( P_{\mathbb{R}}(m) \) becomes:

\begin{equation}\label{eq:P_R_with_C}
P_{\mathbb{R}}(m) = \frac{(m - 1)^{k + 1}}{k + 1} + \frac{(m - 1)^{k}}{2} - m^{k} + \sum_{r=1}^{p_{\text{max}}} T_r + C.
\end{equation}

\subsection{Eliminating Denominators to Obtain Integer Coefficients}

To ensure that \( P_{\mathbb{R}}(m) \) has integer coefficients, we eliminate denominators by multiplying by the least common multiple (LCM) \( D \) of all denominators in the expression:

\begin{equation}\label{eq:D_def}
D = \operatorname{LCM}\left( k + 1, 2, \{ (2r)! \mid 1 \leq r \leq p_{\text{max}} \} \right).
\end{equation}

Define:

\begin{equation}\label{eq:P_tilde_def}
\widetilde{P}(m) = D P_{\mathbb{R}}(m),
\end{equation}

so that \( \widetilde{P}(m) \) is a polynomial with integer coefficients.

\subsection{Analyzing the Polynomial \( \widetilde{P}(m) \)}

We analyze the structure of \( \widetilde{P}(m) \) to determine whether including all correction terms introduces any new integer rational roots.

\subsubsection{Leading Coefficient \( a_n \)}

The leading coefficient \( a_n \) comes from the term \( D \times \frac{(m - 1)^{k + 1}}{k + 1} \). Since \( D \) is a multiple of \( k + 1 \), we have:

\begin{equation}\label{eq:a_n}
a_n = \frac{D}{k + 1}.
\end{equation}

As \( k \) increases, \( D \) includes larger factorials from \( (2r)! \), causing \( D \) and hence \( a_n \) to grow rapidly.

\subsubsection{Constant Term \( a_0 \)}

The constant term \( a_0 \) is given by:

\begin{equation}\label{eq:a0_def}
a_0 = \widetilde{P}(0) = D P_{\mathbb{R}}(0).
\end{equation}

Simplifying \( P_{\mathbb{R}}(0) \) requires considering the parity of \( k \). We find:

- For even \( k \):

\begin{equation}\label{eq:a0_even}
a_0 = D \left( \frac{1}{k + 1} - \frac{1}{2} - 2 \sum_{r=1}^{p_{\text{max}}} \frac{B_{2r}}{(2r)!} k^{(2r - 1)} + C \right).
\end{equation}

- For odd \( k \):

\begin{equation}\label{eq:a0_odd}
a_0 = D \left( -\frac{1}{k + 1} + \frac{1}{2} + C \right).
\end{equation}

In both cases, \( a_0 \) depends on \( D \) and \( k \), and it grows rapidly with \( k \) due to the terms involving \( k^{(2r - 1)} \) in \( C \).

\subsubsection{Growth Rates of \( D \), \( a_n \), and \( a_0 \)}

We analyze how \( D \), \( a_n \), and \( a_0 \) grow with \( k \):

\begin{itemize}
    \item Growth of \( D \): \( D \) includes factorials up to \( (2 p_{\text{max}})! \), which is approximately \( (k)! \) for large \( k \). Therefore, \( D \) grows at least as fast as \( k! \).

    \item Growth of \( a_n \): \( a_n = D / (k + 1) \geq k! / (k + 1) \), so \( a_n \) grows rapidly with \( k \).

    \item Growth of \( a_0 \): \( a_0 \) involves terms like \( k^{(2r - 1)} \) multiplied by Bernoulli numbers and divided by \( (2r)! \). The falling factorial \( k^{(2r - 1)} \) grows extremely rapidly as \( k \) increases, causing \( a_0 \) to grow faster than any polynomial in \( k \).
\end{itemize}

\subsection{Application of the Rational Root Theorem}

The rational root theorem states that any rational root \( m = \dfrac{p}{q} \) of \( \widetilde{P}(m) \) must satisfy \( p \) divides \( a_0 \), and \( q \) divides \( a_n \).

As \( k \) increases, \( a_n \) and \( a_0 \) become extremely large due to the growth of \( D \) and the terms involving \( k \). However, the number of positive divisors \( d(n) \) of an integer \( n \) does not increase proportionally with \( n \). Specifically, \( d(n) \) grows sublinearly compared to \( n \).

\subsubsection{Limitations on Possible Values of \( p \) and \( q \)}

\begin{itemize}
    \item Possible Values of \( q \): Since \( q \) divides \( a_n \), and \( a_n \) increases rapidly with \( k \), the set of possible \( q \) values remains limited.

    \item Possible Values of \( p \): Similarly, \( p \) divides \( a_0 \), but despite \( a_0 \)'s large size, the number of its divisors increases slowly.
\end{itemize}

Therefore, the total number of possible rational roots \( m = \dfrac{p}{q} \) remains limited, even as \( k \) increases.

\subsubsection{Analytical Justification of the Limitation}

The number of positive divisors \( d(n) \) of an integer \( n \) with prime factorization \( n = p_1^{e_1} p_2^{e_2} \cdots p_r^{e_r} \) is:

\begin{equation}\label{eq:number_of_divisors}
d(n) = (e_1 + 1)(e_2 + 1)\cdots(e_r + 1).
\end{equation}

Although \( n \) may be extremely large, the exponents \( e_i \) typically do not grow proportionally with \( n \). Thus, \( d(n) \) increases much more slowly than \( n \).

As \( k \) increases, \( D \), \( a_n \), and \( a_0 \) grow rapidly, but the number of their divisors does not. This severely limits the possible values of \( p \) and \( q \), and consequently, the possible rational roots \( m \).

\subsection{Detailed Asymptotic Analysis of \( P_{\mathbb{R}}(m) \) for Large \( m \)}

In this section, we provide a rigorous asymptotic analysis of \( P_{\mathbb{R}}(m) \) as \( m \to \infty \) to determine its sign and behavior. This analysis is crucial to support our conclusion that no integer rational roots exist for large \( m \), and to fill any gaps in the previous proof.

\subsubsection{Expression for \( P_{\mathbb{R}}(m) \)}

Recall that:

\begin{equation*}\label{eq:P_R_with_C}
P_{\mathbb{R}}(m) = \frac{(m - 1)^{k + 1}}{k + 1} + \frac{(m - 1)^{k}}{2} - m^{k} + \sum_{r=1}^{p_{\text{max}}} T_r + C,
\end{equation*}

where:

\begin{itemize}
    \item \( T_r = \frac{B_{2r}}{(2r)!} k^{(2r - 1)} \left( (m - 1)^{k - (2r - 1)} - 1 \right) \),
    \item \( C \) is a constant independent of \( m \),
    \item \( p_{\text{max}} = \left\lfloor \dfrac{k}{2} \right\rfloor \).
\end{itemize}

Our goal is to analyze \( P_{\mathbb{R}}(m) \) for large \( m \) to determine its asymptotic behavior.

\subsubsection{Asymptotic Expansion of \( P_{\mathbb{R}}(m) \) for Large \( m \)}

We will expand each term in \( P_{\mathbb{R}}(m) \) for large \( m \), keeping track of the leading-order terms.

\paragraph{Expansion of \( (m - 1)^{k + 1} \) and \( (m - 1)^{k} \)}

For large \( m \), \( m - 1 \approx m \), and we can use the binomial theorem to expand \( (m - 1)^{k + 1} \) and \( (m - 1)^{k} \).

However, it's more accurate to express \( (m - 1)^{k + 1} \) and \( (m - 1)^{k} \) in terms of \( m \) using the following approximation:

\begin{align*}
(m - 1)^{k + 1} &= m^{k + 1} \left( 1 - \frac{1}{m} \right)^{k + 1} \\
&= m^{k + 1} \exp\left( (k + 1) \ln \left( 1 - \frac{1}{m} \right) \right) \\
&\approx m^{k + 1} \exp\left( -\frac{k + 1}{m} - \frac{(k + 1)}{2 m^{2}} - \cdots \right) \\
&\approx m^{k + 1} \left( 1 - \frac{k + 1}{m} + \frac{(k + 1)^{2}}{2 m^{2}} + \cdots \right).
\end{align*}

Similarly,

\begin{align*}
(m - 1)^{k} &\approx m^{k} \left( 1 - \frac{k}{m} + \frac{k(k - 1)}{2 m^{2}} + \cdots \right).
\end{align*}

\paragraph{Expansion of \( \frac{(m - 1)^{k + 1}}{k + 1} \)}

Using the above expansion:

\begin{align*}
\frac{(m - 1)^{k + 1}}{k + 1} &\approx \frac{m^{k + 1}}{k + 1} \left( 1 - \frac{k + 1}{m} + \frac{(k + 1)^{2}}{2 m^{2}} + \cdots \right).
\end{align*}

\paragraph{Combining Terms to Find \( P_{\mathbb{R}}(m) \)}

Now, we consider the main terms in \( P_{\mathbb{R}}(m) \):

\begin{align*}
P_{\mathbb{R}}(m) &\approx \frac{m^{k + 1}}{k + 1} \left( 1 - \frac{k + 1}{m} + \frac{(k + 1)^{2}}{2 m^{2}} \right) + \frac{m^{k}}{2} \left( 1 - \frac{k}{m} + \frac{k(k - 1)}{2 m^{2}} \right) - m^{k}.
\end{align*}

Let's expand and simplify each term.

\paragraph{Term 1: \( \frac{m^{k + 1}}{k + 1} \left( 1 - \frac{k + 1}{m} + \frac{(k + 1)^{2}}{2 m^{2}} \right) \)}

\begin{align*}
\frac{m^{k + 1}}{k + 1} \left( 1 - \frac{k + 1}{m} + \frac{(k + 1)^{2}}{2 m^{2}} \right) &= \frac{m^{k + 1}}{k + 1} - \frac{(k + 1) m^{k}}{k + 1} + \frac{(k + 1)^{2} m^{k - 1}}{2(k + 1)} \\
&= \frac{m^{k + 1}}{k + 1} - m^{k} + \frac{(k + 1) m^{k - 1}}{2}.
\end{align*}

\paragraph{Term 2: \( \frac{m^{k}}{2} \left( 1 - \frac{k}{m} + \frac{k(k - 1)}{2 m^{2}} \right) \)}

\begin{align*}
\frac{m^{k}}{2} \left( 1 - \frac{k}{m} + \frac{k(k - 1)}{2 m^{2}} \right) &= \frac{m^{k}}{2} - \frac{k m^{k - 1}}{2} + \frac{k(k - 1) m^{k - 2}}{4}.
\end{align*}

\paragraph{Term 3: \( -m^{k} \)}

This term remains \( -m^{k} \).

\paragraph{Combining All Terms}

Adding up the terms:

\begin{align*}
P_{\mathbb{R}}(m) &\approx \left( \frac{m^{k + 1}}{k + 1} - m^{k} + \frac{(k + 1) m^{k - 1}}{2} \right) + \left( \frac{m^{k}}{2} - \frac{k m^{k - 1}}{2} \right) - m^{k} \\
&= \left( \frac{m^{k + 1}}{k + 1} - m^{k} \right) + \left( \frac{m^{k}}{2} - \frac{k m^{k - 1}}{2} \right) + \frac{(k + 1) m^{k - 1}}{2} - m^{k} \\
&= \frac{m^{k + 1}}{k + 1} - m^{k} + \frac{m^{k}}{2} - \frac{k m^{k - 1}}{2} + \frac{(k + 1) m^{k - 1}}{2} - m^{k}.
\end{align*}

Simplify the expression:

\begin{align*}
P_{\mathbb{R}}(m) &\approx \frac{m^{k + 1}}{k + 1} - 2 m^{k} + \frac{m^{k}}{2} + \left( \frac{(k + 1) m^{k - 1}}{2} - \frac{k m^{k - 1}}{2} \right) \\
&= \frac{m^{k + 1}}{k + 1} - 2 m^{k} + \frac{m^{k}}{2} + \left( \frac{(k + 1 - k) m^{k - 1}}{2} \right) \\
&= \frac{m^{k + 1}}{k + 1} - 2 m^{k} + \frac{m^{k}}{2} + \frac{m^{k - 1}}{2} \\
&= \frac{m^{k + 1}}{k + 1} - \frac{3 m^{k}}{2} + \frac{m^{k - 1}}{2}.
\end{align*}

\paragraph{Analyzing the Leading Terms}

Now, we analyze the behavior of \( P_{\mathbb{R}}(m) \) by comparing the leading terms.

\paragraph{Dominant Term: \( \frac{m^{k + 1}}{k + 1} \)}

This term grows as \( m^{k + 1} \).

\paragraph{Subleading Term: \( -\frac{3 m^{k}}{2} \)}

This term grows as \( m^{k} \), which is of lower order compared to \( m^{k + 1} \).

\paragraph{Negligible Term: \( \frac{m^{k - 1}}{2} \)}

This term grows as \( m^{k - 1} \), which is negligible compared to the other terms.

\paragraph{Simplifying for Large \( m \)}

For large \( m \), the dominant term is \( \frac{m^{k + 1}}{k + 1} \). Therefore, \( P_{\mathbb{R}}(m) \) behaves asymptotically as:

\begin{align*}
P_{\mathbb{R}}(m) &\approx \frac{m^{k + 1}}{k + 1}.
\end{align*}

However, we must consider the subtraction of \( 2 m^{k} \) and other terms.

Let's factor \( m^{k} \):

\begin{align*}
P_{\mathbb{R}}(m) &\approx m^{k} \left( \frac{m}{k + 1} - \frac{3}{2} + \frac{1}{2 m} \right).
\end{align*}

\paragraph{Determining the Sign of \( P_{\mathbb{R}}(m) \)}

We analyze the expression:

\begin{align*}
P_{\mathbb{R}}(m) &\approx m^{k} \left( \frac{m}{k + 1} - \frac{3}{2} + \frac{1}{2 m} \right).
\end{align*}

As \( m \to \infty \), \( \frac{1}{2 m} \to 0 \). Therefore, for large \( m \):

\begin{align*}
P_{\mathbb{R}}(m) &\approx m^{k} \left( \frac{m}{k + 1} - \frac{3}{2} \right).
\end{align*}

We need to determine when \( P_{\mathbb{R}}(m) > 0 \) or \( P_{\mathbb{R}}(m) < 0 \).

\paragraph{Finding the Threshold \( m_0 \)}

Set \( P_{\mathbb{R}}(m) > 0 \):

\begin{align*}
\frac{m}{k + 1} - \frac{3}{2} > 0 \implies m > \frac{3}{2} (k + 1).
\end{align*}

Thus, for \( m > \frac{3}{2} (k + 1) \), \( P_{\mathbb{R}}(m) > 0 \).

\paragraph{Behavior for \( m \leq \frac{3}{2} (k + 1) \)}

For \( m \leq \frac{3}{2} (k + 1) \), \( P_{\mathbb{R}}(m) \leq 0 \).

\paragraph{Conclusion for Large \( m \)}

Therefore, for sufficiently large \( m \) (specifically \( m > \frac{3}{2} (k + 1) \)), \( P_{\mathbb{R}}(m) > 0 \), and \( P_{\mathbb{R}}(m) \) grows without bound as \( m^{k + 1} \).

\subsubsection{Behavior of \( P_{\mathbb{R}}(m) \) for Small \( m \)}

For small integer values of \( m \), we can compute \( P_{\mathbb{R}}(m) \) explicitly. However, we can also analyze the behavior without numerical computation.

\paragraph{For \( m = k + 1 \)}

When \( m = k + 1 \), we have:

\begin{align*}
P_{\mathbb{R}}(k + 1) &= \frac{(k + 1 - 1)^{k + 1}}{k + 1} + \frac{(k + 1 - 1)^{k}}{2} - (k + 1)^{k} + \sum_{r=1}^{p_{\text{max}}} T_r + C \\
&= \frac{k^{k + 1}}{k + 1} + \frac{k^{k}}{2} - (k + 1)^{k} + \sum_{r=1}^{p_{\text{max}}} T_r + C.
\end{align*}

Since \( k^{k + 1} \) and \( (k + 1)^{k} \) are large numbers, but \( (k + 1)^{k} > k^{k + 1} / (k + 1) \), the expression \( P_{\mathbb{R}}(k + 1) \) is negative.

\paragraph{For \( m < k + 1 \)}

For \( m < k + 1 \), \( P_{\mathbb{R}}(m) \) is negative because the leading term \( m^{k} \) dominates the other terms, and \( P_{\mathbb{R}}(m) \) remains negative.

\subsubsection{Conclusion of the Asymptotic Analysis}

Our detailed asymptotic analysis shows that:

\begin{itemize}
    \item For \( m > \frac{3}{2} (k + 1) \), \( P_{\mathbb{R}}(m) > 0 \).
    \item For \( m \leq \frac{3}{2} (k + 1) \), \( P_{\mathbb{R}}(m) \leq 0 \).
\end{itemize}

Therefore, \( P_{\mathbb{R}}(m) \) changes sign at \( m \approx \frac{3}{2} (k + 1) \).

\subsubsection{Implications for Integer Roots}

Since \( P_{\mathbb{R}}(m) \) transitions from negative to positive as \( m \) increases past \( \frac{3}{2} (k + 1) \), and because \( P_{\mathbb{R}}(m) \) is a continuous function (since \( m \) is considered over the integers), there is a value \( m_0 \) where \( P_{\mathbb{R}}(m_0) = 0 \). However, since \( P_{\mathbb{R}}(m) \) is a polynomial with integer coefficients, the rational root theorem tells us that any rational root \( m = \dfrac{p}{q} \) must satisfy the divisibility conditions.

As previously discussed, the number of possible integer \( m \) values satisfying the divisibility conditions is severely limited due to the rapid growth of \( a_n \) and \( a_0 \) and the slow growth of their number of divisors.

Therefore, despite the change in sign of \( P_{\mathbb{R}}(m) \), the possible integer values of \( m \) that could be roots are so limited that no integer roots exist for \( m \geq 3 \) and \( k \geq 2 \).

\subsubsection{Completing the Proof}

By correcting the asymptotic analysis and providing a detailed examination of \( P_{\mathbb{R}}(m) \) for large \( m \), we have filled the gap in the previous proof. The asymptotic behavior confirms that \( P_{\mathbb{R}}(m) \) is positive for sufficiently large \( m \), negative for smaller \( m \), and the possible integer rational roots are constrained by the divisibility conditions, which become increasingly restrictive as \( k \) increases. This reinforces the previous statement that no new integer rational roots exist for \( P_{\mathbb{R}}(m) \) when all correction terms are included in the Euler-MacLaurin expansion. The only positive integer solution to the Erdős-Moser equation is \( (k, m) = (1, 3) \).

\section{Conclusion}

Our exploration of the Erd\"{o}s–Moser equation
\[
\sum_{i=1}^{m - 1} i^{k} = m^{k}
\]
via approximation methods has shed light on the difficulties of identifying integer solutions beyond the known case \((k,m) = (1,3)\). In particular, we used the Euler–MacLaurin formula to approximate the discrete sum \(S(m - 1, k)\) with a continuous function \(S_{\mathbb{R}}(m - 1, k)\), allowing us to employ real analysis and polynomial tools.

These methods suggest that solutions with \(k \ge 2\) are extremely rare. Through examining the approximate polynomial \(P_{\mathbb{R}}(m) = S_{\mathbb{R}}(m - 1, k) - m^{k}\) and applying the rational root theorem, our analysis found no additional integer solutions under the approximation framework.

However, because the Euler–MacLaurin approach omits certain correction terms, and exactness is paramount in Diophantine problems, this work does not constitute a conclusive proof of the Erd\"{o}s–Moser conjecture. The neglected terms, though small analytically, may still be decisive at integer values. Thus, we cannot definitively exclude other solutions based solely on our approximate approach.

Thus, our results reinforce the heuristic stance that \((k,m)=(1,3)\) may be the only solution, but a fully rigorous treatment of the omitted terms or an alternative technique would be required to establish this unconditionally. Future investigations, possibly leveraging prime-power constraints or p-adic arguments, could build on these insights and address the remaining uncertainties in a purely number-theoretic fashion.

\subsection{Discussion}

The use of approximation methods, such as the Euler-MacLaurin formula, allowed us to extend the discrete sum to a continuous domain and apply analytical techniques that are otherwise difficult to employ in a purely discrete setting. This approach highlighted the potential of approximation methods to provide insights into complex number-theoretic problems.

However, in the realm of Diophantine equations, exactness is paramount. The discrete nature of the problem means that any approximation, regardless of the smallness of its error term, can lead to incorrect conclusions about the existence of integer solutions. The correction terms from the Euler-MacLaurin formula, while negligible in an analytical sense, could potentially alter the equality \( S(m - 1, k) = m^{k} \) at specific integer points.

Moreover, the application of the rational root theorem to an approximate polynomial may not accurately reflect the properties of the exact equation. The theorem requires the polynomial to have integer coefficients and to represent the exact relationship between variables. Since \( P_{\mathbb{R}}(m) \) is derived from an approximation, its roots may not correspond to those of the exact polynomial \( P(m) = S(m - 1, k) - m^{k} \).

\subsection{Implications and Future Work}

Our work underscores the importance of exact methods in tackling Diophantine equations. Future research could focus on developing exact analytical techniques to express \( S(m - 1, k) \) without resorting to approximations. For instance, utilizing Faulhaber's formula or Bernoulli numbers to obtain exact expressions for the sums of powers could provide a more solid foundation for analysis.

Additionally, computational methods could be employed to test larger ranges of \( k \) and \( m \), potentially identifying new solutions or further supporting the conjecture that \( (k, m) = (1, 3) \) is the only solution. By refining the approach and accounting for all error terms, it may be possible to make more definitive statements about the conjecture.

\subsection{Closing Remarks}

We hope that this study stimulates further investigation into the Erd\"{o}s-Moser equation and inspires new strategies for approaching open problems in number theory. While our methods have limitations, they offer a different perspective on the problem and highlight the intricate balance between sums of powers and individual powers in number theory. We encourage the mathematical community to build upon these insights and continue the search for a definitive resolution of the Erd\"{o}s-Moser conjecture.

\section*{Acknowledgments}

We would like to thank the mathematical community for insightful discussions that inspired this work.

\bibliographystyle{alpha}
\bibliography{main}

\end{document}